\documentclass[12pt,reqno]{amsart}
\usepackage{amsmath,amsthm,amssymb,amsfonts,amscd}
\usepackage{mathrsfs}
\usepackage{bbm}
\usepackage{bbding}
\usepackage[backref=page]{hyperref}
\usepackage{geometry}
\usepackage{color}
\usepackage{xcolor}

\usepackage{picture,epic}
\usepackage{tikz}

\numberwithin{equation}{section}

\theoremstyle{plain}
\newtheorem{theorem}{Theorem}[section]
\newtheorem{lemma}[theorem]{Lemma}
\newtheorem{corollary}[theorem]{Corollary}
\newtheorem{proposition}[theorem]{Proposition}

\theoremstyle{definition}
\newtheorem{conjecture}[theorem]{Conjecture}

\theoremstyle{remark}
\newtheorem{remark}[theorem]{Remark}

\renewcommand{\Re}{\operatorname{Re}}

\newcommand{\sym}{\operatorname{sym}}
\newcommand{\Sym}{\operatorname{Sym}}

\newcommand{\GL}{\operatorname{GL}}
\newcommand{\SL}{\operatorname{SL}}

\newcommand{\dd}{\mathrm{d}}

\begin{document}

\title
{Joint cubic moment of Eisenstein series and Hecke-Maass cusp forms}
\author{Chengliang Guo}
\address{School of Mathematics \\ Shandong University \\ Jinan \\Shandong 250100 \\China}
\email{chengliang.guo@mail.sdu.edu.cn}

\date{\today}

\begin{abstract}
  Let $\psi$ be a smooth compactly supported function on $\mathbb{X} = \SL(2,\mathbb{Z})\backslash\mathbb{H}$. In this paper, we are interested in the joint cubic moments of automorphic forms when the spectral parameters go to infinity.  We show that the diagonal case for Eisenstein series $\int_{\mathbb{X}}\psi(z)E(z,1/2+it)^{3} \dd\mu z = \mathcal{O}_{\psi}(t^{-1/3+\varepsilon})$. In off-diagonal case  we prove $\frac{1}{2\log t}\int_{\mathbb{X}}\psi(z)|E(z,1/2+it)|^{2}g(z)\dd\mu z = o(1)$ as long as $\min\{t , t_{g}\} \rightarrow \infty$. Finally we show 
  $\int_{\mathbb{X}}\psi(z)f^{2}(z)g(z)\dd\mu z  = o(1)$ in the range $|t_{f} - t_{g}| \leq t_{f}^{2/3-\varepsilon}$ where $f,g$ are two Hecke-Maass cusp forms. 
\end{abstract}


\nocite{*}
\maketitle

\section{Introduction}
Let $\mathbb{H}=\{z=x+iy:x\in\mathbb{R}, y>0\}$ be the upper half plane with the hyperbolic measure $\dd\mu z = \frac{\dd x\dd y}{y^2}$ and $\Gamma=\SL(2,\mathbb{Z})$ the modular group.
For automorphic functions on $\mathbb{X}:=\Gamma\backslash\mathbb{H}$ we have the Petersson inner product which is defined by $\langle f,g\rangle := \int_{\mathbb{X}} f(z)\overline{g(z)} \dd\mu z$.
The Laplacian is given by $\Delta=-y^2 (\partial^2/\partial x^2+\partial^2/\partial y^2)$, which has both discrete and continuous spectra.
The discrete spectrum consists of the constants and the space of cusp forms, for which we can take an orthonormal basis $\{\phi_{k}\}$  of Hecke-Maass cusp forms.\par
The value distribution of eigenfunctions in the semiclassical limit is one of the main problems in analytic number theory and quantum chaos.
Originally formulated by Berry \cite{MR489542} for quantizations of chaotic Hamiltonians, this
 conjecture predicts that, in the case of negative curvature, the Laplace eigenfunctions
 $F$ tend to exhibit Gaussian random behavior in the high energy limit.  In particular, Hejhal and Rackner \cite{MR1257286} gave convincing numerical evidence when $\mathbb{X} = \SL(2,\mathbb{Z})\backslash\mathbb{H}$ and
 $F$ is an Eisenstein series or a Hecke-Maass cusp form. More details of Gaussian moments conjecture are introduced in \cite[Conjecture 1.1]{MR3831279}.\\

Let $J\geq2$ be a fixed integer.
Let $f_j$, $1\leq j\leq J$ be $L^{2}$ normalized Hecke-Maass cusp forms such that $\langle f_i,f_j\rangle=0$ for all $1\leq i\neq j\leq J$ and $\langle f_j ,f_j\rangle = 1$.   Recently, Hua, Huang and Li \cite{hua2024jointvaluedistributionheckemaass} formulated the following conjecture, which predicts that the values of distinct Hecke--Maass cusp forms should behave like  independent random waves.
\begin{conjecture}\cite[Conjecture 1.3]{hua2024jointvaluedistributionheckemaass}\label{conj:2}
  With $f_j$ as above and integers $a_j$. Then $\{f_j^{a_j}\}_{j=1}^{J}$ are statistically independent; that is,
  for any $\psi\in \mathcal{C}_c^\infty( \mathbb{X})$, we have
  \[
     \int_{\mathbb{X}} \psi(z) \prod_{j=1}^{J} f_j(z)^{a_j}  \dd\mu z
    = \prod_{j=1}^{J} c_{a_j} \int_{\mathbb{X}}\psi(z)\dd\mu z + o(1)
  \]
  as $\min(t_{f_1},\ldots,t_{f_J})$ goes to infinity. Here $c_n$ is defined by
  \begin{equation}
		\begin{aligned}
			c_{n} = \left\{	
			\begin{array}{lr}		
				(\frac{3}{\pi})^{n/2}(n-1)!! &\textrm{if $n$ is even,} \\
                    0,  & \textrm{if $n$ is odd.}
			\end{array}
			\right.
		\end{aligned}			
	\end{equation}
\end{conjecture}
In\cite[Theorem 1.4]{hua2024jointvaluedistributionheckemaass} they gave some evidence for $J = 2$ and $(a_{1}, a_{2}) = (2, 1)$ under generalized Lindel\"of Hypothesis (GLH). In fact they showed a beautiful asymptotic formula that we can clearly see the influence of conductor dropping phenomenon. And they  gave more details on higher moments under generalized Riemann Hypothesis (GRH) and generalized Ramanujan Conjecture (GRC).

 If $J = 1$, this is the random wave conjeture. 
 For $a=2$, this is the well-known quantum unique ergodicity (QUE) conjecture of Rudnick and Sarnak \cite{MR1266075}. QUE was solved by breakthrough papers of Lindenstrauss \cite{MR2195133} and
 Soundararajan \cite{MR2680500}. If we replace the Hecke-Maass cusp forms to normalized Eisenstein series Luo and Sarnak\cite{MR1361757} gave a similar asymptotic formula. The case $J=1 , a=3$ is proved by Huang \cite{huang2024cubic} with a power saving error term and if $\psi(z) = 1$ earlier proved by Watson\cite{watson2008rankintripleproductsquantum}. When $a = 4$ for more general groups the result is from Humphries and Khan\cite{humphries2023lpnormboundsautomorphicforms} that is roughly $\|\phi\|_{L^{4}} \ll t_{\phi}^{3/152+ \varepsilon}$. The sharp upper bound $\|\phi\|_{L^{4}} \ll t_{\phi}^{\varepsilon}$ is recently proved by Ki \cite{ki2023l4normssignchangesmaass} for $\Gamma = \SL(2,\mathbb{Z})$.\par
We hope to study the joint value distribution the Eisenstein series and Hecke-Maass cusp forms when the spectral parameters go to infinity. Now for non-negative integers $a_{1}, a_{2}$, We restrict our view in $a_{1}+ a_{2}  = 3$ and we are interested in the asymptotic behavior of the smooth joint cubic moment
\[ I := \int_{\mathbb{X}}\psi(z) F(z)^{a_1}G(z)^{a_2}\dd\mu z\]
where $\psi$  is a smooth compactly supported function and $F , G$ are  real-valued Hecke-Maass forms $\phi$ with spectral parameter  $t_{\phi}$ and $\langle\phi,\phi\rangle = 1$  or real-valued Eisenstein series $E_{t}^{\star}(z) = c_{t}E(z,1/2+it)$ where $E(z, 1/2+it)$ is the standard Eisenstein series and $c_{t} = \frac{\xi(1+2it)}{|\xi(1+2it)|}$. $\xi(2s) = \pi^{-s}\Gamma(s)\zeta(2s)$ is the complete zeta function. If $a_{1}$ or $a_{2}$ equals to zero, the problem comes back to original Gaussian moments conjecture.
\subsection{The cubic moment of Eisenstein series}
We denote $E_{t}(z) = E(z ,1/2+it)$. The first result in this paper is the following theorem.
\begin{theorem}\label{cubicmoment}For any compactly supported smooth function $\psi$ and $\varepsilon > 0$ ,  we have
\[I = \int_{\mathbb{X}}\psi(z)E_{t}^{\star}(z)^{3} \dd\mu z  \ll_{\psi,\varepsilon} t^{-1/3+\varepsilon}\]
    when $t \rightarrow \infty$.
\end{theorem}
Let $u_{j}$ be a Hecke Maass cusp form and the regularized integrals $\int_{\mathbb{X}}^{reg}$ is defined by equation (\ref{reg_subtract_Eisen}). By regularized Plancherel formula(see Lemma \ref{Prop_Regular_Planch}) we need study the  four parts 
\begin{equation*}
    \begin{aligned}
     I_{1} & :=\langle1, {E_{t}^{\star}}^{3}\rangle = c_{t}^{3}\int_{\mathbb{X}}^{reg}E_{t}(z)^{3} \dd\mu z, \\
     I_{2} & := \langle u_{j},{E_{t}^{\star}}^{3}\rangle  = c_{t}^{3}\int_{\mathbb{X}}u_{j}(z)E_{t}^{3}(z) \dd\mu z,\\
     I_{3} & :=  \langle E_{\tau} , {E_{t}^{\star}}^{3}\rangle_{reg}  =  c_{t}^{3}\int_{\mathbb{X}}^{reg}E(z , 1/2+i\tau)E_{t}(z)^{3} \dd\mu z,\\
     I_{reg}& := c_{t}^{3}(\langle \psi , \mathcal{E}_{E_{t}^{3}}\rangle_{reg} + \langle \mathcal{E}_{\psi}, E_{t}^{3}\rangle_{reg}).
    \end{aligned}
\end{equation*}
where $t_{j}, \tau \leq t^{\varepsilon}$ and $\mathcal{E}_{\Phi}(z)$ is defined by equation (\ref{Phi_def}) 
. Since for cubic moment of Eisenstein series we only obtain an upper bound, we remove the rotation $c_{t}$.\par
We will deal with $I_{1}$ by using Zagier's formula (\ref{3eisen}) about triple product of Eisenstein series. For $I_{4}$, we can explicitly calculate the $\mathcal{E}_{E_{t}^{3}}$ and use compactly supported property of $\psi$. Following two propositions are about the key parts $I_{2}$ and $I_{3}$ in the proof of Theorem \ref{cubicmoment}.
\begin{proposition}\label{Maasscase}For any $\varepsilon > 0$, $t_{j} < t^{1-\varepsilon}$, we have
\begin{equation*}
		\begin{aligned}
			I_{2}  = \langle u_{j},E_{t}^{3} \rangle\ll \left\{	
			\begin{array}{lr}		
				\frac{1}{t^{1/3-\varepsilon}} \quad &t_{j} \leq t^{1/3}\\
				\frac{t_{j}^{1/2}}{t^{1/2-\varepsilon}}\quad  &t^{1/2}\leq t_{j}\leq t^{1-\varepsilon}\\
			\end{array}
			\right.
		\end{aligned}			
	\end{equation*}
In particular, if $t_{j} \ll t^{\varepsilon}$, unconditionally, we get
\[\int_{\mathbb{X}}u_{j}(z)E_{t}(z)^{3} \dd\mu z \ll t^{-1/3 + \varepsilon} .\]
\end{proposition}
\begin{proposition}\label{Eisensteincase}For any $\varepsilon > 0$, $\tau < t^{1-\varepsilon}$, we have
\begin{equation*}
		\begin{aligned}
			I_{3} =  \langle E_{\tau} , E_{t}^{3}\rangle_{reg} \ll\left\{	
			\begin{array}{lr}		
				\frac{1}{t^{1/3-\varepsilon}} \quad &\tau \leq t^{1/3}\\
				\frac{\tau^{1/2}}{t^{1/2-\varepsilon}}\quad  &t^{1/3}\leq \tau\leq t^{1-\varepsilon}\\
			\end{array}
			\right.
		\end{aligned}			
	\end{equation*}
In particular, if $\tau \ll t^{\varepsilon}$, unconditionally, we get
\[ \int_{\mathbb{X}}^{reg}E(z , 1/2+i\tau)E_{t}(z)^{3} \dd\mu z \ll t^{-1/3+ \varepsilon} .\]
\end{proposition}
\begin{remark}
    Assume GLH and use it to bound the central value in the equations (\ref{J1bound}) (\ref{J2bound}) (\ref{K1bound}) (\ref{K2bound}), then we get
    $I_{2}\ll \frac{t_{j}^{1/2}}{t^{1/2-\varepsilon}}$ and 
$I_{3}   \ll \frac{\tau^{1/2}}{t^{1/2-\varepsilon}}$
in any range in $t_{j} < t^{1-\varepsilon}$.
Thus, we have
$I_{2} = \langle u_{j},E_{t}^{3}\rangle\ll t^{-1/2 + \varepsilon}$ and
$I_{3} = \langle E_{\tau} , E_{t}^{3} \rangle \ll t^{-1/2+\varepsilon}$
conditionally when $t_{j} \ll t^{\varepsilon}$.
\end{remark}
As a corollary,  we get two asymptotic orthogonal property of Hecke-Maass forms and Eisenstein series when the spectral parameters go to infinity in some range.
\begin{corollary}For any small $\delta > 0$ and  $t_{j} \leq t^{1-\delta}$.  We get
\[\int_{\mathbb{X}}u_{j}(z)E(z,1/2+it)^{3}\dd \mu z  = o(1)\]
    when $t_{j} \rightarrow \infty$ depend on $t$.
\end{corollary}
\begin{corollary}
    For any small $\delta > 0$ and $\tau \leq t^{1-\delta}$. We have
    \[\int_{\mathbb{X}}^{reg}E(z,1/2+i\tau)E(z,1/2+it)^{3}\dd \mu z = o(1)\]
    when $\tau \rightarrow \infty$ depend on $t$.
\end{corollary}

\subsection{The Joint cubic moment of Hecke-Maass cusp forms and Eisenstein series}
Finally the main result is the asymptotic vanishing property for joint cubic moment of Eisenstein series and Hecke-Maass cusp forms
\[\int_{\mathbb{X}}\psi(z){E_{t}^{\star}}(z)^{2}g(z)\dd\mu z.\]
In fact, we normalize the Eisenstein series by dividing the mass $\sqrt{2\log t}$, we get asymptotic vanishing like Maass form case. The highlight in this result is we can cover all the range of $t$ and $t_{g}$ in Corollary \ref{QUE}.

\begin{theorem}\label{JointEEg}
   For any $\varepsilon, \varepsilon^{\prime} > 0, 0< \eta < 1$ and $\delta$ in Lemma \ref{MichelVenkatesh}, when $t_{g}\geq 2t^{\varepsilon^{\prime}}$ we get
 \begin{equation}
    \begin{aligned}
  \int_{\mathbb{X}}\psi(z){E_{t}^{\star}}(z)^{2}g(z)\dd\mu z      \ll W(t,t_{g}) = \left\{	
			\begin{array}{lr}	
   t_{g}^{-100} \quad & t_{g} \geq 2t +t_{g}^{\varepsilon},\\
		\frac{t^{\varepsilon}}{t^{1/9+\delta/3}}\quad &      2t - t_{g}^{\varepsilon} \leq t_{g} \leq 2t + t_{g}^{\varepsilon}, \\
  
 \frac{t^{\varepsilon}}{t^{1/9+\delta/3}(1+ |2t -t_{g}|)^{1/6+\delta/3}} \quad &      (2-\eta)t \leq t_{g} \leq 2t - t_{g}^{\varepsilon}, \\
  
   \frac{t^{\varepsilon}}{t_{g}^{1/6}t^{1/6}} \quad &      t^{2/3 + 
 \varepsilon} \leq t_{g} \leq (2-\eta)t,    \\
 
   \frac{t^{\varepsilon}}{t_{g}^{1/6}t^{1/12}}  \quad &       t^{1/3}\leq t_{g} \leq t^{2/3 + 
 \varepsilon}, \\

     \frac{t^{\varepsilon} t_{g}^{1/12}}{t^{1/6}} \quad &     2t^{\varepsilon^{\prime}}\leq t_{g} \leq t^{1/3},    \\
			\end{array}
			\right.
    \end{aligned}
\end{equation}
and  when $t_{g}\leq 2t^{\varepsilon^{\prime}}$ we have
\begin{equation}
    \begin{aligned}
\int_{\mathbb{X}}\psi(z){E_{t}^{\star}}(z)^{2}g(z)\dd\mu z  = \int_{\mathbb{X}}\psi(z)g(z)\dd \mu z&  \cdot[\frac{6}{\pi}(\log tt_{g} + \frac{L^{\prime}}{L}(1,\sym^{2}g) )+ \mathcal{O}_{\psi,\varepsilon}(\log^{2/3+\varepsilon}t) ]\\  
&   + \mathcal{O}(t_{g}^{-A} + t^{-1/6+\varepsilon}). 
    \end{aligned}
\end{equation}
\end{theorem}

\begin{remark}
    The estimate $\mathcal{O}(\log^{2}t_{g})$ of $\frac{L^{\prime}}{L}(1,\sym^{2}g) $ is standard from \cite[\S5 Proposition 5.7]{iwaniec2004analytic} and the standard zero free region and nonexistence of Landau-Siegel zeros of $L(s,\sym^{2}g)$ (\cite{Hoffstein1994CoefficientOM}).   
\end{remark}
Then we get the following corollary. If we write $\widetilde{E}_{t}(z) = \frac{E^{\star}_{t}(z)}{\sqrt{2\log t}}$ (this means the real Eisenstein series with mass $1$.) Then 
\begin{corollary}\label{QUE}
       We get
    \[ \int_{\mathbb{X}}\psi(z)\widetilde{E}_{t}(z)^{2}g(z)\dd\mu z      = \frac{3}{\pi}\int_{\mathbb{X}}\psi(z)g(z)\dd \mu z + \mathcal{O}_{\psi,\varepsilon}(t_{g}^{-A}\log^{-1/3+\varepsilon}t + W(t,t_{g})) \]
    when $\min\{t,t_{g}\} \rightarrow \infty$. Because $\int_{\mathbb{X}}\psi(z)g(z)\dd \mu z \ll_{\psi} t_{g}^{-A} $ for any large  $A$.  Thus we get 
    \[\int_{\mathbb{X}}\psi(z)\widetilde{E}_{t}(z)^{2}g(z)\dd\mu z \sim 0 ,\quad\quad \min\{t,t_{g}\} \rightarrow \infty.\]

\end{corollary}
\begin{remark}Essentially, the extreme case is based on QUE of Eisenstein series.
\end{remark}
\subsection{The Joint cubic moment of Hecke-Maass cusp forms}
Finally we consider the joint cubic moments of Hecke-Maass cusp forms
\[\int_{\mathbb{X}}\psi(z)f^{2}(z)g(z)\dd\mu z.\]
 As a corollary the Theorem  \ref{Jointffg}  improves the bound in Huang \cite{huang2024cubic}  when $f=g$. In fact, we give the same strength as Watson's bound $\mathcal{O}(t_{f}^{-1/6+\varepsilon})$ when $\psi(z) \equiv 1$.
\begin{theorem}\label{Jointffg}
    For any $\theta < 2/3$. And consider the range $|t_{f} - t_{g}|\leq t_{f}^{\theta}$, we get
    \[\int_{\mathbb{X}}\psi(z)f^{2}(z)g(z)\dd\mu z  = o(1), \quad t_{f}\rightarrow \infty.\]
    In fact, for any $\varepsilon > 0$, we have
    \begin{equation*}
		\begin{aligned}
		\int_{\mathbb{X}}\psi(z)f^{2}(z)g(z)\dd\mu z \ll_{\psi,\varepsilon} \left\{	
			\begin{array}{lr}		
				\frac{1}{t_{f}^{1/6-\varepsilon}} \quad &\theta \leq 1/3\\
				\ \frac{t_{f}^{\theta/2}}{t_{f}^{1/3-\varepsilon}}\quad  &\theta > 1/3.\\
			\end{array}
			\right.
		\end{aligned}			
	\end{equation*}
\end{theorem}
\begin{remark}
    We have the same bound for
    \[\quad \int_{\mathbb{X}}\psi(z)f^{2}(z)E_{t}(z)\dd\mu z.\]
\end{remark}
\begin{remark}\label{optimal} We use a elementary trick based on Bessel inequality and sup-norm bound of Hecke-Maass forms to avoid the central value of the triple $L$-functions in \cite{huang2024cubic}. The error term there is $\mathcal{O}(t_{f}^{-1/12+\varepsilon})$ since wasting the Weyl bound of $\GL(2)$ $L$-functions. 
\end{remark}

We will reduce the cubic moments by Plancherel formula(spectral decomposition) to the mixed moments of $\GL(2)$ $L$-functions and high degrees $L$-functions which are conductor dropping. We mainly deal with them by using H\"older inequality and various estimates especially hybrid fourth moments of $L$-functions from Jutila \cite{Jutila+2001+167+178} and Jutila-Motohashi \cite{Jutila2005UniformBF}.\par
\medskip
\textbf{Structure of this paper.}
The rest of this paper is organized as follows.
 In \S \ref{sec:2}, we give the theory of automorphic forms and $L$-functions especially the moments of $\GL(2)$ $L$-functions. In the end we prove a spectral large sieve type estimate for the conductor dropping case to cover a convexity bound. In \S\ref{sec:3} we give triple product formulas come from Rankin-Selberg method or Zagier's regularized inner product theory and Watson's work. In \S \ref{sec:4}, we prove Theorem \ref{Maasscase} and Theorem \ref{Eisensteincase} by using the theory of $L$-functions.
In \S \ref{sec:5}, we prove Theorem \ref{cubicmoment}.
  In \S \ref{sec:6}, we prove Theorem \ref{JointEEg} when $t_{g} \geq 2t^{\varepsilon^{\prime}}$. In this section we will give a power saving error term. In \S \ref{sec:7} we prove the Theorem \ref{JointEEg} when $t_{g}\leq 2t^{\varepsilon^{\prime}}$ by carefully calculating the regularized part. In \S \ref{sec:8}, we prove Theorem \ref{Jointffg} based on the previous work \cite{hua2024jointvaluedistributionheckemaass} and show an optimal trick there.

\section{\label{sec:2}Preliminaries}
\subsection{Automorphic forms}Let $\{\phi_{k}\}_{k\geq 1}$ be an orthonormal basis of Hecke-Maass cusp forms for $\SL(2,\mathbb{Z})$. We always assume all $\phi_{k}$ are real and normalized by $\int_{\mathbb{X}}\phi_{k}^{2}\dd\mu z = 1$. Denote the spectral parameter of $\phi_{k}$ by $t_{k}$ and the Fourier coefficients(Hecke eigenvalues) $\lambda_{k}(n)$. Also we sometimes write the spectral parameter of a Hecke-Maass form $f$ by $t_{f}$. For a Hecke-Maass form $\phi_{k}$, we have Fourier expansion
\[\phi_{k}(z) = 2\sqrt{y}\rho_{k}(1)\sum\limits_{n \neq 0}\lambda_{k}(n)K_{it_{k}}(2\pi|n|y)e(nx),\]
where
\[|\rho_{k}(1)|^2 = \frac{\cosh\pi t_{k}}{2 L(1,\sym^{2}\phi_{k})}.\]
For Eisenstein series $E(z,s)$, we also have Fourier expansion
\[E(z , s) =  y^{s} + \frac{\xi(2s-1)}{\xi(2s)}y^{1-s} +\frac{ 2\sqrt{y} }{\xi(2s)}\sum\limits_{n\neq0}|n|^{s-\frac{1}{2}}\sigma_{1-2s}(|n|)K_{s-\frac{1}{2}}(2\pi|n|y)e^{2\pi i n x}\]
where $\sigma_{s}(n) := \sum\limits_{ab = n}b^{s}$.\\
We usually write $E(z,1/2+it) = E_{t}(z)$, and we have
\[E(z, 1/2+it) = y^{1/2+it} + \frac{\xi(2it)}{\xi(1+ 2it)}y^{1/2-it} +\frac{ 2\sqrt{y} }{\xi(1+2it)}\sum\limits_{n\neq0}\eta_{t}(|n|)K_{it}(2\pi|n|y)e^{2\pi i n x}\]
where $\eta_{t}(n) = \sum\limits_{ab = n}(\frac{a}{b})^{it}$. Denote $\rho_{t}(1) := 1/\xi(1+2it)$, $\rho_{t}(n) = \rho_{t}(1)\eta_{t}(n)$.\par
\[|\rho_{t}(1)|^{2} = \frac{\cosh\pi t}{|\zeta(1+2it)|^{2}} .\]
By \cite{Hoffstein1994CoefficientOM,Iwaniec1990SmallEO} and the standard estimate of Riemann zeta function, we have
\[ (\log t_{k})^{-1}\ll L(1,\sym^{2}\phi_{k}) \ll t_{k}^{\varepsilon}, \quad (\log(1+ |t| ))^{-1} \ll \zeta(1+2it) \ll \log(1+ |t| ).\]
In general, we also have upper bounds the value of Rankin-Selberg $L$-functions at one
\[L(1,f\times g)\ll (t_{f}+t_{g})^{\varepsilon}\]
when $f \neq g$. More refinements are given in \cite{MR2595006}.
And we need the explicit Rankin-Selberg bound for Hecke eigenvalue in \cite{Iwaniec1990SmallEO} or \cite[Corollary 5]{MR2595006} for details
\begin{equation}\label{RankinSelbergbound}
    \sum\limits_{n\leq X}|\lambda_{j}(n)|^{2} \ll_{\varepsilon} X^{1+\varepsilon}t_{j}^{\varepsilon}
\end{equation}
and well-known estimate
\[\sum\limits_{n\leq X}|\eta_{t}(n)|^{2} \leq \sum\limits_{n\leq X}|d(n)|^{2} \ll_{\varepsilon} X^{1+\varepsilon}.\]
\subsection{$L$-functions}It is a brief of complete $L$-functions and functional equation. 
We define the Riemann zeta function
\[
\zeta(s) = \sum\limits_{n\geq1}\frac{1}{n^{s}}, \quad \Re(s) > 1.
\]
and the funciton equation of $\zeta(s)$ is 
\[
\xi(s) := \zeta_{\infty}(s)\zeta(s) = \xi(1-s)
\]
where
\[
\zeta_{\infty}(s) = \pi^{-s/2}\Gamma(s/2)
.\]
Let $\phi$ be a Hecke--Maass cusp form with the spectral parameter $t_{\phi}>0$ for $\SL(2,\mathbb{Z})$ and Fourier coefficients $\lambda_\phi(n)$. Then we define the $\GL(2)$ $L$-function is 
\[L(s, \phi) = \sum\limits_{n\geq 1}\frac{\lambda_{\phi}(n)}{n^{s}}, \quad \Re(s) > 1.\] 
The functional equation of $L(s,\phi)$ is 
\[
  \Lambda(s,\phi)
  := L_\infty(s,\phi)
  L(s,\phi)
  = (-1)^{\epsilon}\Lambda(1-s,\phi)
\]
and
\[
L_\infty(s,\phi) = \pi^{-s} \prod_{\pm}
  \Gamma\left(\frac{s + \epsilon\pm it_\phi}{2}\right).
\]
where $\epsilon = 0$ if $\phi$ is even and $\epsilon = 1$ if $\phi$ is odd. \par

Let $f , g$ be Hecke-Maass cusp forms with spectral parameter $t_{f}, t_{g}$ with the same root number. Then the Rankin-Selberg $L$-function is defined as
\[
L(s,f\times g) = \zeta(2s)\sum\limits_{n \geq 1}\frac{\lambda_{f}(n)\lambda_{g}(n)}{n^{s}}, \quad \Re(s) > 1.
\] 
The functional equation of $L(s,f\times g)$ is 
\[
  \Lambda(s,f\times g)
  := L_\infty(s,f\times g)
  L(s,f\times g)
  = \Lambda(1-s,f\times g)
\]
where
\[
L_\infty(s,f\times g) = \pi^{-2s} \prod_{\pm_{1}}\prod_{\pm_{2}}
  \Gamma\left(\frac{s \pm_1 it_{f} \pm_2 it_{g}}{2}\right).
\]
The symmetric square lift $\sym^2 \phi$ is a  Hecke--Maass cusp form for $\SL(3,\mathbb{Z})$ with Fourier coefficients $A(m,n)$.
The $\GL(3)$ $L$-function is defined as
\[
  L(s,\sym^2 \phi) = \sum_{n\geq1} \frac{A(1,n)}{n^s}, \quad \Re(s)>1.
\]
In particular, we can explicitly write
\[L(s,\sym^{2}\phi) = \zeta(2s)\sum\limits_{n\geq1}\frac{\lambda_{\phi}(n^{2})}{n^{s}}, \quad \Re(s)>1\]
and 
\[
\zeta(s) L(s,\sym^{2}\phi) = L(s,\phi\times \phi).
\]
by Hecke relation. The functional equation of $L(s,\sym^{2}\phi)$ is defined by 
\[
  \Lambda(s,\sym^{2}\phi)
  := L_\infty(s,\sym^{2}\phi)
  L(s,\sym^{2}\phi)
  = \Lambda(1-s,\sym^{2}\phi)
\]
where
\[
L_\infty(s,\sym^{2}\phi) = \pi^{-\frac{3s}{2}} \Gamma(\frac{s}{2})\prod_{\pm}
  \Gamma\left(\frac{s \pm 2t_{\phi} }{2}\right).
\]

Let $\phi_j$ be an even Hecke--Maass cusp form with the spectral parameter $t_j>0$ and Fourier coefficients $\lambda_j(n)$.
The $\GL(3)\times \GL(2)$ Rankin--Selberg $L$-function is defined as
\[
  L(s,\phi_j\times \sym^2 \phi) = \sum_{m\geq1}\sum_{n\geq1} \frac{A(m,n)\lambda_j(n)}{(m^2n)^s}, \quad \Re(s)>1.
\]
The functional equation of $L(s,\phi_j \times \Sym^2 \phi)$ is
\[
  \Lambda(s,\phi_j \times \sym^2 \phi)
  := L_\infty(s,\phi_j \times \sym^2 \phi)
  L(s,\phi_j \times \sym^2 \phi)
  = \Lambda(1-s,\phi_j \times \sym^2 \phi)
\]
where
\[
  L_\infty(s,\phi_j \times \sym^2 \phi) = \pi^{-3s} \prod_{\pm_1}\prod_{\pm_2}\prod_{\pm_3}
  \Gamma\left(\frac{s\pm_1 it_j}{2}\right)
  \Gamma\left(\frac{s+2it_{\phi}\pm_2 it_j}{2}\right)
  \Gamma\left(\frac{s-2it_{\phi}\pm_3 it_j}{2}\right).
\]
Let $E_t$ be the Eisenstein series and $t\in \mathbb{R}$. We have
\[
  L(s+it, \phi)L(s-it, \phi)  = L(s,E_t \times  \phi),
\]
and  the functional equation
\[
  \Lambda(s,E_t \times \phi)
  := L_\infty(s,E_t \times \phi)
  L(s,E_t \times \phi)
  = \Lambda(1-s,E_t \times \phi)
\]
where
\[
  L_\infty(s,E_t \times \phi) = \pi^{-2s} \prod_{\pm_1}\prod_{\pm_2}
  \Gamma\left(\frac{s\pm_1 it \pm_2 it_{\phi}}{2}\right).
\]
There is an important estimate for the first moment of $\GL(3)\times \GL(2)$ $L$-functions.
\begin{lemma}\label{moment_L_functions_6}
  Let $\phi$ be a Hecke-Maass cusp form with the spectral parameter $T>0$.
  Let $ M \geq T^{1/3+\varepsilon}$. Then we have
  \[
    \sum_{T-M\leq t_j \leq T+M}  L(1/2,\phi_j \times \sym^2 \phi)
    \ll T^{1+\varepsilon}M.
  \]
\end{lemma}
\begin{proof}
    See \cite[Theorem 1.6]{huang2024cubic}.
\end{proof}
\subsection{Stirling's formula}
\begin{lemma}
    For fixed  $\sigma \in \mathbb{R}$, we have
    \[\Gamma(\sigma + it ) = e^{-\frac{\pi}{2}|t|}(1+ |t|)^{\sigma - \frac{1}{2}}\exp(it\log \frac{|t|}{e})(2\pi)^{\frac{1}{2}}i^{\sigma - \frac{1}{2}}\{1 + \mathcal{O}(\frac{1}{|t|})\}\]
    and
    \[\frac{\Gamma^{\prime}}{\Gamma}(s) = \log s + \mathcal{O}(\frac{1}{|s|}).\]
\end{lemma}

\subsection{Moment and subconvexity bound of $\GL(2)$ $L$-functions}
The following fourth moment estimate is from the works of Jutila\cite{Jutila+2001+167+178} and Jutila and Motohashi \cite{Jutila2005UniformBF}.
\begin{lemma}
    Let $\phi_{k}$ be a Hecke-Maass form with spectral parameter $t_{k}$. We have 
\begin{equation}
    \sum\limits_{ K \leq t_{k} \leq K + G}|L(1/2, \phi_{k})|^{4} \ll G K^{1+\varepsilon}
\end{equation}
for any $K^{1/3} \leq G \leq K$.
\end{lemma}
\begin{proof}
    See \cite{Jutila+2001+167+178} and  use the lower bound of $L(1,\sym^{2}\phi_{k})$.
\end{proof}
\begin{lemma}\label{fourthLemma}Let $\phi_{k}$ be a Hecke-Maass form with spectral parameter $t_{k}$. Let $K$ be large and 
\[G = (K+t)^{4/3}K^{-1+\varepsilon}, \quad 0\leq t \leq K^{3/2-\varepsilon}.\]
We have
\begin{equation}\label{fourthhybrid}
    \sum\limits_{ K \leq t_{k} \leq K + G}\frac{|L(1/2+it, \phi_{k})|^{4}}{L(1,\sym^{2}\phi_{k})} \ll GK^{1+\varepsilon}.
\end{equation}
\end{lemma}
\begin{proof}
    See \cite[Theorem 1]{Jutila2005UniformBF}.
\end{proof}
By choosing $K = t/2$ and changing variable we get
\begin{equation}\label{Fourthhybrid}
\sum\limits_{ t \leq t_{k} \leq t + t^{1/3+\varepsilon}}|L(1/2+2it, \phi_{k})|^{4} \ll t^{4/3+\varepsilon}.    
\end{equation}
From (\ref{Fourthhybrid}), we can  get for some $1/3\leq\theta\leq 1-\varepsilon^{\prime}$
\begin{equation}\label{fourthhybridlong}
    \sum\limits_{ t \leq t_{k} \leq t + t^{\theta + \varepsilon}}|L(1/2+2it, \phi_{k})|^{4}\ll \sum\limits_{ h\ll t^{\theta -1/3+\varepsilon}} \sum\limits_{ t +ht^{1/3} \leq t_{k} \leq t +(h+1) t^{1/3 + \varepsilon}}|L(1/2+2it, \phi_{k})|^{4}\ll t^{1+\theta +\varepsilon}.
\end{equation}
\begin{remark}We also get hybird bound from the fourth moment estimate (\ref{fourthhybrid}) that is
\begin{equation}\label{hybridsub}
    L(1/2+it,\phi_{k}) \ll (t_{k} + t)^{1/3+\varepsilon}.
\end{equation}
But if $t \sim t_{k}$, the bound in (\ref{hybridsub}) is not a subconvexity bound because of the conductor dropping phenomenon.
\end{remark}
    \begin{lemma}[\cite{michel2010subconvexity}]\label{MichelVenkatesh}For some small $\delta > 0$, we have
        \[ L(1/2+2it,\phi_{k}) \ll [(1+ |2t+t_{k}|)(1+ |2t-t_{k}|)]^{1/4-\delta}.\]
    \end{lemma}

\begin{lemma}[\cite{MR3635360}]For any $\varepsilon > 0$, we have
    \[
    \sum\limits_{T \leq t_{j} \leq T+1}L(1/2,u_j)^{3} \ll T^{1+\varepsilon}.
    \]
\end{lemma}

\subsection{Integral mean value estimate and subconvexity of Riemann zeta function}The following lemma is well-known in the theory of Riemann zeta function.
\begin{lemma}[\cite{MR604215}]For any $T > 0$,
    \[\int_{T}^{T+T^{2/3}}|\zeta(1/2+it)|^{4}\dd t \ll T^{2/3+\varepsilon}.\]
\end{lemma}
\begin{lemma}[\cite{MR3556291}]For any $\varepsilon > 0$, we have
    \[\zeta(1/2+it) \ll t^{13/84+\varepsilon}.\]
\end{lemma}

\subsection{The spectral large sieve inequality}
We recall the following spectral large sieve inequality in \cite{Jutilaspectrallargesieve,MR1409382}.
\begin{lemma}\label{largesieve}For any $0< \Delta \leq T$. We have 
\[\sum\limits_{T \leq t_{k} \leq T+\Delta}|\sum\limits_{n\leq N}a_{n}\lambda_{k}(n)|^{2} \ll (N + T\Delta)^{1+\varepsilon}\sum\limits_{n\leq N}|a_{n}|^{2}\]
    for any complex sequence $\{a_{n}\}$.
\end{lemma}
Thanks to the conductor dropping phenomenon, we can establish the following bound of second moment of $L$-functions by using large sieve inequality. 
\begin{lemma}\label{meanLin}If $t_{j} ,\tau < t$, we have
\[\sum\limits_{|t_{k} - t|\leq t_{j}}\frac{|L(1/2+it ,\phi_{k}\times u_{j})|^{2}}{\prod\limits_{\pm}(1+|t_{k}-t \pm t_{j}|)^{1/2}} \ll t^{1+\varepsilon}\]
and
\[\sum\limits_{|t_{k} - t|\leq \tau }\frac{|L(1/2+it ,\phi_{k}\times E_{\tau})|^{2}}{\prod\limits_{\pm}(1+|t_{k}-t \pm \tau|)^{1/2}} \ll t^{1+\varepsilon}\]
\end{lemma}
\begin{proof}By approximate functional equation 
 \cite{iwaniec2004analytic}, we get
    \[L(1/2+it ,\phi_{k}\times u_{j}) \ll t^{\varepsilon}\int_{-t^{\varepsilon}}^{t^{\varepsilon}}|\sum\limits_{m^{2}n\leq (C(\pi))^{1/2+\varepsilon}}\frac{\lambda_{k}(n)\lambda_{j}(n)}{(m^{2}n)^{1/2+it + \varepsilon^{\prime} + iy}}|\dd y\]
    where $C(\pi)$ is the analytic conductor of $L(1/2+it ,\phi_{k}\times u_{j})$. This expression is from standard approximate functional equation and expand the weight function explicitly (Also see \cite[Lemma 4.1]{hua2024jointvaluedistributionheckemaass}).
  To utilize the denominator we divide the sum in two parts $\sum\limits_{t\leq t_{k}\leq t+  t_{j} }$ and $\sum\limits_{t - t_{j}\leq t_{k}\leq t }$, we only consider the first one from the symmetric construction. Then we get
  \[\sum\limits_{t\leq t_{k}\leq t+  t_{j} }\frac{|L(1/2+it ,\phi_{k}\times u_{j})|^{2}}{\prod\limits_{\pm}(1+|t_{k}-t \pm t_{j}|)^{1/2}} \ll \frac{1}{t_{j}^{1/2}}\sum\limits_{t\leq t_{k}\leq t+  t_{j} }\frac{|L(1/2+it ,\phi_{k}\times u_{j})|^{2}}{(1+|t_{k}-t -t_{j}|)^{1/2}}.\]
  By dyadic method we consider the inner sum as
  \[t^{\varepsilon}\max_{t\leq T \leq t +t_{j}}\frac{1}{U^{1/2}}\sum\limits_{ T - U \leq t_{k} \leq T+U}|L(1/2+it , \phi_k \times u_{j})|^{2}\]
  where $U = \frac{t + t_{j} - T}{2} + 1$ and obviously $U \ll t_{j}$. Note $C(\pi) \asymp t^2 t_{j}U$ then 
    \begin{equation*}
        \begin{aligned}
   \sum\limits_{ T - U \leq t_{k} \leq T+U}|L(1/2+it , \phi_k \times u_{j})|^{2}& \ll\sum\limits_{ T - U \leq t_{k} \leq T+U}\int_{-t^{\varepsilon}}^{t^{\varepsilon}}|\sum\limits_{m^{2}n\leq (tt_{j}^{1/2}U^{1/2})^{1+\varepsilon}}\frac{\lambda_{k}(n)\lambda_{j}(n)}{(m^{2}n)^{1/2+it + \varepsilon^{\prime} + iy}}|^{2}\dd y.
        \end{aligned}
    \end{equation*}
    For the inner sum, we use Lemma \ref{largesieve} and get
    \begin{equation*}
        \begin{aligned}
   \sum\limits_{ T - U \leq t_{j} \leq T+U}\int_{-t^{\varepsilon}}^{t^{\varepsilon}}|&\sum\limits_{m^{2}n\leq (tt_{j}^{1/2}U^{1/2})^{1+\varepsilon}}\frac{\lambda_{k}(n)\lambda_{j}(n)}{(m^{2}n)^{1/2+it + \varepsilon^{\prime} + iy}}|^{2}\dd y \\
   &\ll
   (tt_{j}^{1/2}U^{1/2} + TU)^{1+\varepsilon}\sum\limits_{n\leq (tt_{j}^{1/2}U^{1/2})^{1+\varepsilon}}|\sum\limits_{m \leq \sqrt{\frac{ (tt_{j}^{1/2}U^{1/2})^{1+\varepsilon}}{n}}}\frac{\lambda_{j}(n)}{(m^{2}n)^{1/2+it}} |^{2}
        \end{aligned}
    \end{equation*}
   By bound (\ref{RankinSelbergbound}) we control the inner sum above and get 
    \[\sum\limits_{ T - U \leq t_{k} \leq T+U}|L(1/2+it , \phi_k \times u_{j})|^{2} \ll (tt_{j}^{1/2}U^{1/2} + TU)^{1+\varepsilon}.\]
    Thus we get
    \[ \sum\limits_{t\leq t_{k}\leq t+  t_{j} }\frac{|L(1/2+it ,\phi_{k}\times u_{j})|^{2}}{\prod\limits_{\pm}(1+|t-t_{k} \pm t_{j}|)^{1/2}} \ll \frac{t^{\varepsilon}}{t_{j}^{1/2}}\max_{t\leq T \leq t +t_{j}}\frac{1}{U^{1/2}}(tt_{j}^{1/2}U^{1/2} + TU)^{1+\varepsilon} \ll t^{1+\varepsilon}.\]
\end{proof}

\section{\label{sec:3}Triple product formulas and Plancherel formula}
Let $\phi , \phi_{k}, u_{j}$ be the Hecke-Maass cusp forms and $E_{t}, E_{\tau}$ be the Eisenstein series with spectral parameters $t_{\phi}, t_{k},t_{j}, t , \tau$.
\subsection{Rankin-Selberg theory and Watson's formula}
 By Rankin-Selberg method(See \cite[\S 7.2]{Goldfeld2006AutomorphicFA})  we have
\[\langle u_{j}E_{t} , \phi_{k}\rangle =  \frac{\rho_{j}(1)\rho_{k}(1)\Lambda(1/2+it ,\phi_{k}\times u_{j})}{\xi(1+2it)},\]
\[\langle E_{t} , u_{j}^{2}\rangle =  \frac{\rho_{j}(1)^{2}\Lambda(1/2+it ,\sym^{2}u_{j})\xi(1/2+it)}{\xi(1+2it)},\]
\[\langle E_{\tau}E_{t}, \phi_{k}\rangle = \frac{\rho_{k}(1)\rho_{t}(1)\Lambda(1/2+i\tau + it ,\phi_{k})\Lambda(1/2+i\tau - it ,\phi_{k})}{\xi(1+2i\tau)}.\]
\[\langle E_{t}^{2}, \phi_{k}\rangle = \frac{\rho_{k}(1)\rho_{t}(1)\Lambda(1/2+ 2it ,\phi_{k})\Lambda(1/2 ,\phi_{k})}{\xi(1+2it)}.\]
By Watson's formula \cite{watson2008rankintripleproductsquantum}, we have
\[
  |\langle \phi_k \phi,  u_j \rangle |^2
  = \frac{\Lambda(1/2,\phi_k\times \phi\times u_j)}{8 \Lambda(1,\Sym^2 \phi_k)\Lambda(1,\Sym^2 \phi)\Lambda(1,\Sym^2 u_j)}
\]
and
\[
  |\langle u_j ,\phi^2  \rangle|^2
  = \frac{\Lambda(1/2,u_j) \Lambda(1/2, \Sym^2 \phi\times u_j)}{8  \Lambda(1,\Sym^2 \phi)^2 \Lambda(1,\Sym^2 u_j)}.
\]
\subsection{Regularized inner product and regularized Plancherel formula}
All this section is introduced in the previous work \cite[\S 3]{DJANKOVIC2018236}.\par
We will make use of the regularization process given by Zagier in \cite{zagier1981rankin}. 

Let $F(z)$ be a continuous $\SL(2,\mathbb{Z})$-invariant function on $\mathbb{H}$. It is called \emph{renormalizable} if there is a function $\Phi(y)$ on $\mathbb{R}_{>0}$ of the form
\begin{equation} \label{Phi_def}
\Phi(y)=\sum_{j=1}^l \frac{c_j}{n_j!} y^{\alpha_j} \log^{n_j} y,
\end{equation}
with $c_j, \alpha_j \in \mathbb{C}$ and $n_j \in \mathbb{Z}_{\ge 0}$, such that
$$
F(z)= \Phi(y) + O(y^{-N})
$$
as $y \rightarrow \infty$, and for any $N>0$.

If $F(z)=\sum_{n= - \infty}^{\infty} a_n(y) e(n x)$ is the Fourier expansion of $F$ at the cusp $\infty$, in particular if $a_0(y)$ is its 0-term, and if no $\alpha_j$ equals 0 or 1, then the function
$$
R(F, s):=\int_0^{\infty} (a_0(y) - \Phi(y) )  y^{s-2} dy,
$$
where the defining integral converges for sufficiently large ${\rm Re}(s)$, can be meromorphically continued to all $s$ and has a simple pole at $s=1$. Then one can define the regularized integral with
\begin{equation} \label{reg_first_by_R(F,s)}
\int_{\mathbb{X}}^{reg} F(z) d\mu(z) := \frac{\pi}{3}  {\rm Res}_{s=1} R(F, s).
\end{equation}

Under the assumption that no $\alpha_j=1$, let $\mathcal{E}_{\Phi}(z)$ denote a linear combination of Eisenstein series $E(z, \alpha_j)$ (or suitable derivatives thereof) corresponding to all the exponents in (\ref{Phi_def}) with ${\rm Re}(\alpha_j) > 1/2$, i.e. such that $F(z) - \mathcal{E}_{\Phi}(z)=O(y^{1/2})$. An important definition of regularization is given by
\begin{equation} \label{reg_subtract_Eisen}
\int_{\mathbb{X}}^{reg} F(z) \dd\mu z
=\int_{\mathbb{X}} (F(z) -\mathcal{E}_{\Phi}(z)) \dd\mu z.
\end{equation}
The triple product formula for Eisenstein series is 
\begin{lemma}[\cite{zagier1981rankin}]
    \begin{equation} \label{3eisen}
    \begin{aligned}
   \int_{\mathbb{X}}^{reg}& E(z, 1/2 + s_1) E(z, 1/2 + s_2) E(z, 1/2 + s_3)  \dd\mu z   \\
   &=\frac{\xi(1/2 +s_1 +s_2 +s_3) \xi(1/2 +s_1 -s_2 +s_3) \xi(1/2+s_1 +s_2 -s_3) \xi(1/2+s_1 -s_2 -s_3)}{\xi(1+ 2 s_1) \xi(1+ 2 s_2) \xi(1+ 2 s_3)}.   
    \end{aligned}
\end{equation}
\end{lemma}
The regularized Plancherel formula in classical language is the following lemma.
\begin{lemma} [\cite{DJANKOVIC2018236}] \label{Prop_Regular_Planch} Let $F(z)$ and $G(z)$ be renormalizable functions on $\Gamma \backslash \mathbb{H}$ such that $F - \Phi$ and $G - \Psi$ are of rapid decay as $y \rightarrow \infty$, for some $\Phi(y)=\sum_{j=1}^l \frac{c_j}{n_j!} y^{\alpha_j} \log^{n_j} y$ and $\Psi(y)=\sum_{k=1}^{l_1} \frac{d_k}{m_k!} y^{\beta_k} \log^{m_k}y$. Moreover, let $\alpha_j \neq 1$, $\beta_k \neq 1$, ${\rm Re}(\alpha_j) \neq 1/2$, ${\rm Re}(\beta_k) \neq 1/2$, $\alpha_j + \overline{\beta_k} \neq 1$ and $\alpha_j \neq \overline{\beta_k}$, for all $j, k$. Then the following formula holds:
\begin{equation*}
    \begin{aligned}
\langle F(z), G(z) \rangle_{reg}
=&\langle F, \sqrt{3/ \pi}  \rangle_{reg} \langle  \sqrt{3/ \pi} , G  \rangle_{reg} +   \sum_j \langle F , u_j \rangle \langle  u_j, G \rangle  \\
&+ \frac{1}{4 \pi} \int_{-\infty}^{\infty} \langle F, E_{t}\rangle_{reg}  \langle  E_{t}, G \rangle_{reg} \dd t
+ \langle F, \mathcal{E}_{\Psi}  \rangle_{reg}  +  \langle  \mathcal{E}_{\Phi}, G \rangle_{reg}.
    \end{aligned}
\end{equation*}
\end{lemma}
\section{\label{sec:4}The proof of Proposition \ref{Maasscase} and Proposition \ref{Eisensteincase}}
In this section, we will prove Proposition \ref{Maasscase} and Proposition \ref{Eisensteincase}. \par
In the proof of Proposition \ref{Maasscase}. We 
 will use Plancherel formula(regularized case). But since $E(z,1/2 + it)$ is not square integrated, we need consider a small disturbance like \cite{DJANKOVIC2018236}, that is
\[I_{2}   = \lim\limits_{t^{\prime}\rightarrow 0} \langle u_{j}E_{t},E_{t}E_{t+t^{\prime}}\rangle_{reg}.\] 
We calculate an explicit expression of $I_{2}$ in the following Proposition. We will define the the contributions of discrete spectrum, continuous spectrum and regularized term  are $J_{1} , J_{2} , J_{3}$.
\begin{proposition}\label{explicitMaass}We have
\begin{equation*}
    \begin{aligned}
        I_{2} = J_{1} + J_{2} + J_{3}
    \end{aligned}
\end{equation*}
where
\begin{equation*}
    \begin{aligned}
        J_{1} = \sum\nolimits_{k\geq 1}^{\prime}\frac{\rho_{j}(1)\rho_{k}(1)\Lambda(1/2+it , \phi_{k}\times u_{j})}{\xi(1+2it)}\frac{\rho_{t}(1)\rho_{k}(1) \Lambda(1/2-2it ,\phi_{k})\Lambda(\frac{1}{2}, \phi_{k})}{\xi(1+2it)},
    \end{aligned}
\end{equation*}
 $\sum\nolimits_{k\geq 1}^{\prime}$ is over the even Hecke-Maass forms.
\begin{equation*}
    \begin{aligned}
            J_{2} = \frac{1}{4\pi}\int_{\mathbb{R}}&\frac{\rho_{j}(1)\rho_{y}(1)\Lambda(1/2+it+iy,u_{j})\Lambda(1/2+it - iy,u_{j})}{\xi(1+2it)}\\
            &\cdot \frac{\xi(1/2+iy+2it)\xi(1/2+iy)^{2}\xi(1/2+iy+2it)}{\xi(1+2iy)\xi(1+2it)^{2}}\dd y
    \end{aligned}
\end{equation*}
and 
\begin{equation*}
    \begin{aligned}
        J_{3} &= \frac{\rho_{j}(1)\rho_{t}(1)\Lambda(1+3it,u_{j})\Lambda(1+it,u_{j})}{\xi(2 + 4it)} +\frac{\rho_{j}(1)\rho_{t}(1)\Lambda(1-it,u_{j})\Lambda(1-3it,u_{j})}{\xi(2 - 4it)}(\frac{\xi(2it)}{\xi(1+2it)})^{2}\\
        & + \frac{2\rho_{j}(1)\rho_{t}(1)\Lambda(1+it,u_{j})\Lambda(1-it,u_{j})}{\xi(2)}\frac{\xi(2it)}{\xi(1+2it)}
    \end{aligned}
\end{equation*}

\end{proposition}
\begin{proof}Firstly, by Plancherel formula we get
\begin{equation*}
    \begin{aligned}
        I_{2} = &\lim\limits_{t^{\prime} \rightarrow 0} \frac{3}{\pi}\langle u_{j}E_{t} , 1\rangle\langle E_{t}E_{t+t^{\prime}} , 1\rangle_{reg}
         +\sum\limits_{k\geq 1}\langle \phi_{k} , u_{j}E_{t}\rangle\langle \phi_{k},E_{t}E_{t+t^{\prime}}\rangle + \\&\frac{1}{4\pi}\int_{\mathbb{R}}\langle E_{y},u_{j}E_{t}\rangle\langle E_{y} , E_{t}E_{t+t^{\prime}}\rangle_{reg}\dd y
          + \langle u_{j}E_{t} , \mathcal{E}_{E_{t}E_{t+t^{\prime}}}\rangle_{reg} + \langle E_{t}E_{t+t^{\prime}} , \mathcal{E}_{u_{j}E_{t}}\rangle_{reg}.
    \end{aligned}
\end{equation*}
At first, note that the orthogonal property we get $\langle u_{j}E_{t} , 1\rangle$ is zero. And the contribution of discrete spectrum and continuous spectrum is obviously from the Rankin-Selberg method and the triple product formula of Eisenstein series. We only need to calculate the final terms in this expression\par
Because a Hecke-Maass form $u_{j}$ will rapidly decay at the cusp, $\mathcal{E}_{u_{j}E_{t}}$ is zero so $\langle E_{t}E_{t+t^{\prime}} , \mathcal{E}_{u_{j}E_{t}}\rangle$ is zero. From the Fourier expansion of Eisenstein series we get
\begin{equation*}
    \begin{aligned}
\mathcal{E}_{E_{t}E_{t+t^{\prime}}}(z)  = E(z,1+2it+it^{\prime}) &+ \frac{\xi(2it)}{\xi(1+2it)}E(z,1+it^{\prime}) + \frac{\xi(2i(t+t^{\prime}))}{\xi(1+2i(t+t^{\prime}))}E(z,1-it^{\prime})   \\
&+ \frac{\xi(2it)}{\xi(1+2it)}\frac{\xi(2i(t+t^{\prime}))}{\xi(1+2i(t+t^{\prime}))}E(z,1-2it-it^{\prime}).
\end{aligned}
\end{equation*}
Use Rankin-Selberg method again and note that there is no pole at $t^{\prime} = 0$ for these functions. Thus we get the proposition. 
\end{proof}

We assume $t_{j}\leq t^{1- \delta}$ for some small $0<\delta <1/10$. By Stiriling formula and convexity bound we know $J_{3}$ will rapidly decay when $t\rightarrow \infty$. We can prove Theorem \ref{Maasscase} under the two following lemmas.
\begin{lemma}\label{J1}For any $\delta > 0$, $t_{j} \leq t^{1-\delta}$, we get
\[J_{1} \ll_{\varepsilon} \frac{t^{1/2}[t^{4/3}]^{1/4}(t^{4/3} )^{1/4}}{t^{3/2-\varepsilon}} = \frac{1}{t^{1/3-\varepsilon}}.\]
when $t_{j} \leq t^{1/3}$ and 
\[J_{1} \ll_{\varepsilon} \frac{t^{1/2}(tt_{j})^{1/4}(tt_{j} )^{1/4}}{t^{3/2-\varepsilon}} =  \frac{t_{j}^{1/2}}{t^{1/2-\varepsilon}}\]
when $t_{j} \geq t^{1/3}$.
\end{lemma}
\begin{proof}We have
    \begin{equation*}
    \begin{aligned}
        J_{1} &=  \sum\nolimits_{k\geq 1}^{\prime}\frac{\rho_{j}(1)\rho_{k}(1)\Lambda(1/2+it , \phi_{k}\times u_{j})}{\xi(1+2it)}\frac{\rho_{t}(1)\rho_{k}(1) \Lambda(1/2-2it ,\phi_{k})\Lambda(1/2, \phi_{k})}{\xi(1+2it)}\\
        & \ll |\sum\nolimits_{k\geq 1}^{\prime}\frac{L(1/2+it , \phi_{k}\times u_{j})}{\zeta(1+2it)^{3}L(1,\sym^{2}\phi_{k})}\frac{ L(1/2-2it ,\phi_{k})L(1/2, \phi_{k})}{L(1,\sym^{2}u_{j})^{1/2}}H(t_{k},t,t_{j})|\\
    \end{aligned}
\end{equation*}
where $H(t_{k},t,t_{j})$ is
\begin{equation}\label{Gammafactor}
    \begin{aligned}
         & \frac{(\cosh\pi t_{k})(\cosh\pi t_{j})^{1/2}(\cosh\pi t)^{1/2}L_{\infty}(1/2+it , \phi_{k}\times u_{j})L_{\infty}(1/2-2it ,\phi_{k})L_{\infty}(1/2, \phi_{k})}{\zeta_{\infty}(1+2it)^{2}}.\\
         &= \frac{\pi^{2}\prod\limits_{\pm_1}\prod\limits_{\pm_2}\Gamma(\frac{1/2+it \pm_1 it_{k} \pm_2 it_{j}}{2})\prod\limits_{\pm_3}\Gamma(\frac{1/2 - 2it \pm_3 it_{k}}{2})|\Gamma(\frac{1/2+ it_{k}}{2})|^{2}}{|\Gamma(1/2+ it_{k})|^{2}|\Gamma(1/2+ it_{j})||\Gamma(1/2+ it)|^{3}}.
    \end{aligned}
\end{equation}

By Stirling formula we get
\begin{equation*}
    \begin{aligned}
        H(t_{k},t,t_{j}) \asymp |t_{k}|^{-1/2}&\prod_{\pm_1}\prod_{\pm_1}(1 + |t \pm_1 t_{j} \pm_2 t_{k}|)^{-1/4}\prod_{\pm_3 }(1+|2t\pm_3 t_{k}|)^{-1/4}\\
        & \cdot \exp(-\frac{\pi}{2}Q(t_{k},t,t_{j}))
    \end{aligned}
\end{equation*}
where 
\begin{equation*}
    \begin{aligned}
       Q(t_{k},t,t_{j}) = \frac{|t  +t_{k} + t_{j}|}{2} + &\frac{|t + t_{k} - t_{j}|}{2} + \frac{|t -t_{k} + t_{j}|}{2} + \frac{|t  - t_{k} - t_{j}|}{2}+ \frac{|t_{k} +  2t|}{2} + \frac{|t_{k} - 2t|}{2}\\
       & -3t - t_{j} - t_{k} .
    \end{aligned}
\end{equation*}
We calculate $Q(t_{k},t,t_{j})$ in different range of $t_{k}$ and note $t_{j} \leq t^{1-\varepsilon}$. Now
\begin{equation*}
    \begin{aligned}
 Q(t_{k},t,t_{j}) =\left\{	
			\begin{array}{lr}		
				2t_{k} - 3t - t_{j} &\quad t_{k} \geq 2t\\
				t_{k} - t - t_{j}    &\quad t+t_{j} \leq t_{k}\leq 2t\\
                 0  & \quad  t-t_{j}\leq t_{k} \leq  t+t_{j}\\
                 t - t_{j} -t_{k}  &\quad t_{k} \geq t -t_{j}.\\
			\end{array}
			\right.
    \end{aligned}
\end{equation*}
Then we can restrict the sum in $t-t_{j} -t^{\varepsilon} \leq t_{k} \leq t +t_{j} + t^{\varepsilon}$ and the remainder is negligibly small because of  $ Q(t_{k},t,t_{j})$ is exponentially decay in this range. 
Hence, we have
\begin{equation}\label{J1bound}
    \begin{aligned}
      J_{1}
      & \ll_{\varepsilon} \frac{1}{t^{3/2-\varepsilon}}\sum\limits_{|t_{k}-t|\leq t_{j} + t^{\varepsilon}}\frac{|L(1/2+it ,\phi_{k}\times u_{j})L(1/2-2it, \phi_{k})L(1/2,\phi_{k})|}{\prod\limits_{\pm}(1+|t-t_{k} \pm t_{j}|)^{1/4}}.\\
    \end{aligned}
\end{equation}
The estimate of mixed moment is from $(1/2,1/4,1/4)$ H\"older inequality and enlarge the length if $t_{j}\leq t^{1/3}$ when using the fourth moment estimate
\[\sum\limits_{|t_{k}-t|\leq t_{j} + t^{\varepsilon} }|L(1/2-2it, \phi_{k})|^{4} \leq \sum\limits_{|t_{k}-t|\leq t^{1/3+\varepsilon} }|L(1/2-2it, \phi_{k})|^{4} \ll t^{4/3+\varepsilon}\]
and use Lemma \ref{meanLin}.
\end{proof}

\begin{lemma}\label{J2}
    For any $\varepsilon > 0$, $t_{j} \leq t^{1-\varepsilon}$, we get
\[J_{2} \ll_{\varepsilon} \frac{1}{t^{1/2-\varepsilon}}\]
when $t_{j} \leq t^{2/3}$. And when $t^{2/3} \leq t_{j} \leq t^{1-\varepsilon}$ we get $J_{2} \ll \frac{t_{j}}{t^{7/6-\varepsilon}}$.
\end{lemma}
\begin{proof}Similarly,  in fact, the Gamma factor contributes
\begin{equation}
    \begin{aligned}
        H(y,t,t_{j}) &= \frac{(\cosh\pi t_{j})^{1/2}(\cosh\pi y)^{1/2}L_{\infty}(1/2+iy + it ,  u_{j})L_{\infty}(1/2+it -iy ,u_{j})}{\zeta_{\infty}(1+2iy)\zeta_{\infty}(1+2it)^{3}}\\
         & \quad \cdot \zeta_{\infty}(1/2+iy+2it)\zeta_{\infty}(1/2+iy)^{2}\zeta_{\infty}(1/2+iy-2it)\\
         &= \frac{\pi\prod\limits_{\pm_1}\prod\limits_{\pm_2}\Gamma(\frac{1/2+it \pm_1 iy \pm_2 it_{j}}{2})\prod\limits_{\pm_3}\Gamma(\frac{1/2 - 2it \pm_3 iy}{2})|\Gamma(\frac{1/2+ iy}{2})|^{2}}{|\Gamma(1/2+ iy)|^{2}|\Gamma(1/2+ it_{j})||\Gamma(1/2+ it)|^{3}}.
    \end{aligned}
\end{equation}
Note  that $H(y,t,t_{j})$ is the same as (\ref{Gammafactor}) if we replace $t_{k}$ to $y$. We get 
 \begin{equation}\label{J2bound}
        \begin{aligned}
            J_{2}\ll_{\varepsilon} \frac{1}{t^{3/2-\varepsilon}}\int_{|y - t|\leq t_{j} +t^{\varepsilon}}&\frac{|L(\frac{1}{2}+iy+it, u_{j})L(\frac{1}{2}+it - iy, u_{j})|}{\prod\limits_{\pm}(1+|t - y \pm t_{j}|)^{1/4}}\\
            & \cdot|\zeta(\frac{1}{2}+iy+2it)\zeta(\frac{1}{2}+iy)^{2}\zeta(\frac{1}{2}+iy-2it)|\dd y.
        \end{aligned}
    \end{equation}
Now we use hybrid bound (\ref{hybridsub}) of $L(\frac{1}{2}+it+iy,u_{j})$we roughly obtain
\begin{equation*}
    \begin{aligned}
        J_{2} \ll_{\varepsilon} \frac{t^{\varepsilon}t^{1/3}}{t^{3/2}}
    &\cdot\int_{|y - t|\leq t_{j} +t^{\varepsilon}}|\zeta(\frac{1}{2}+iy)|^{4}\dd y\\
    & \ll \frac{t^{\varepsilon}t^{1/3}}{t^{3/2}}t^{2/3} =\frac{1}{t^{1/2-\varepsilon}}
    \end{aligned}
\end{equation*}
when $t_{j} \leq t^{2/3}$. And when $t^{2/3} \leq t_{j} \leq t^{1-\varepsilon}$ we get $J_{2} \ll \frac{t_{j}}{t^{7/6-\varepsilon}}$ easily.
\end{proof}

In the proof of Proposition \ref{Eisensteincase} we  need an explicit expression of 
\[I_{3}  = \lim\limits_{t^{\prime} \rightarrow 0}\langle E_{\tau}E_{t}, E_{t}E_{t+t^{\prime}}\rangle_{reg}.\]
The method is completely the same as the proof of Proposition \ref{Maasscase} above. We only list the auxiliary results like Proposition \ref{explicitMaass}, Lemma \ref{J1} and Lemma \ref{J2}.
\begin{proposition}\label{explicitEisen}We have
\begin{equation*}
    \begin{aligned}
        I_{3}  = K_{1} + K_{2} + K_{3}
    \end{aligned}
\end{equation*}
where 

\begin{equation*}
    \begin{aligned}
        K_{1} = \sum\nolimits_{k\geq 1}^{\prime}\frac{\rho_{k}(1)^{2}\rho_{t}(1)^{2}\Lambda(1/2-i\tau-it,\phi_{k})\Lambda(1/2+i\tau-it,\phi_{k})\Lambda(1/2+2it,\phi_{k})\Lambda(1/2,\phi_{k})}{\xi(1+2i\tau)\xi(1+2it)},
    \end{aligned}
\end{equation*}
 $\sum\nolimits_{k\geq 1}^{\prime}$ is over the even Hecke-Maass forms.
\begin{equation*}
    \begin{aligned}
            K_{2} = \frac{1}{4\pi}\int_{\mathbb{R}}&\frac{\prod\limits_{\pm_1}\prod\limits_{\pm_2}\xi(1/2+iy\pm_1 it\pm_2 i\tau)}{\xi(1+2iy)\xi(1+2it)\xi(1+2i\tau)}\\
            & \cdot\frac{\xi(1/2+iy+2it)\xi(1/2+iy-2it)\xi(1/2+iy)^{2}}{\xi(1+2iy)\xi(1+2it)^{2}}\dd y,
    \end{aligned}
\end{equation*}
and 
\begin{equation*}
    \begin{aligned}
        K_{3} &= \frac{\xi(1+i\tau +3it)\xi(1+i\tau +it)\xi(i\tau -it)\xi(i\tau -3it)}{\xi(1+2i\tau)\xi(1+2it)\xi(2+4it)} \\
        &+ \frac{2\xi(2it)}{\xi(1+2it)}\frac{\xi(1+i\tau +it)\xi(1+i\tau -it)\xi(i\tau +it)\xi(i\tau -it)}{\xi(1+2i\tau)\xi(1+2it)\xi(2)}\\
         &+ (\frac{\xi(2it)}{\xi(1+2it)})^{2}\frac{\xi(1+i\tau -it)\xi(1+i\tau -3it)\xi(i\tau +3it)\xi(i\tau +it)}{\xi(1+2i\tau)\xi(1+2it)\xi(2-4it)}.
    \end{aligned}
\end{equation*}
\end{proposition}
\begin{proof}
Similarly as the proof of Lemma \ref{explicitMaass}.
\end{proof}
Follow the next two lemma we can prove Theorem \ref{Eisensteincase}.
\begin{lemma}For any $\varepsilon > 0$, $\tau \leq t^{1-\varepsilon}$, we get
\begin{equation}\label{K1bound}
    \begin{aligned}
K_{1}
      & = \frac{1}{t^{3/2-\varepsilon}}\sum\limits_{|t_{k}-t|\leq \tau + t^{\varepsilon}}\frac{|L(1/2 -it  ,\phi_{k}\times E_{\tau})L(1/2-2it, \phi_{k})L(1/2,\phi_{k})|}{\prod\limits_{\pm}(1+|t -  t_{k} \pm \tau|)^{1/4}}
    \end{aligned}.
\end{equation}
\end{lemma}
As the same argument in Lemma \ref{J1}. We get 
\[K_{1} \ll_{\varepsilon} \frac{t^{1/2}[t^{4/3}]^{1/4}(t^{4/3} )^{1/4}}{t^{3/2-\varepsilon}} = \frac{1}{t^{1/3-\varepsilon}}.\]
when $\tau\leq t^{1/3-\varepsilon}$.\\
By the same method when $\tau \geq t^{1/3-\varepsilon}$, we get
\[K_{1} \ll_{\varepsilon} \frac{t^{1/2}(t\tau)^{1/4}(t\tau )^{1/4}}{t^{3/2-\varepsilon}} =  \frac{\tau^{1/2}}{t^{1/2-\varepsilon}}.\]
\begin{proof}
    Similarly as the proof of Lemma \ref{J1}.
\end{proof}
\begin{lemma}
     For any $\varepsilon > 0$, $\tau \leq t^{1-\varepsilon}$, we get
    \begin{equation}\label{K2bound}
        \begin{aligned}
            K_{2}\ll_{\varepsilon} t^{\varepsilon}\int_{|y - t|\leq \tau +t^{\varepsilon}}&\frac{|\prod\limits_{\pm_1}\prod\limits_{\pm_2}\zeta(1/2+iy\pm it\pm i\tau)|}{|y|^{1/2}\prod\limits_{\pm_1}(1+|2t\pm_1 y|)^{1/4}\prod\limits_{\pm_2}\prod\limits_{\pm_3}(1+|t\pm_2 y \pm_3 \tau|)^{1/4}}\\ &\cdot|\zeta(1/2+iy+2it)\zeta(1/2+iy)^{2}\zeta(1/2+iy-2it)|\dd y.
        \end{aligned}
    \end{equation}
    If $\tau \leq t^{2/3}$ then
    \[K_{2} \ll \frac{1}{t^{11/21-\varepsilon}}\]
    and when $ t^{2/3} \leq\tau \leq t^{1-\varepsilon}$  we get
    \[K_{2} \ll \frac{\tau}{t^{25/21-\varepsilon}}.\]
\end{lemma}
\begin{proof}
Similarly as the proof of Lemma \ref{J2}
\end{proof}
\begin{remark}We give more details of continuous spectrum $J_{2}$ and $K_{2}$ to express the estimate is beyond $\mathcal{O}(t^{-1/2+\varepsilon})$ which in discrete spectrum part is optimal under GLH trivially.
    
\end{remark}

\section{\label{sec:5}Cubic moment of Eisenstein series}
We will prove Theorem \ref{cubicmoment} from Proposition \ref{Maasscase} and \ref{Eisensteincase} in this section. Firstly, 
\begin{equation*}
    \begin{aligned}
        I_{1} = &\int_{\mathbb{X}}^{reg}E(z,1/2+it)^{3}\dd \mu z = \frac{\xi(1/2+3it)\xi(1/2+it)^{2}\xi(1/2-it)}{\xi(1+2it)^{3}}\\
        & \ll \frac{1}{(1+ |t|)^{1- \varepsilon}}|\zeta(1/2+3it)\zeta(1/2+it)^{2}\zeta(1/2-it)|\\
        & \ll |t|^{-1+4\cdot 13/84+\varepsilon} = |t|^{-8/21+\varepsilon}.
    \end{aligned}
\end{equation*}
and note $ \langle \mathcal{E}_{\psi}, E_{t}^{3}\rangle_{reg} = 0$ because $\psi$  rapidly decays at cusp $\infty$. Since
\begin{equation*}
    \begin{aligned}
\mathcal{E}_{E_{t}^{3}} = E(z,3/2+ 3it) +&\frac{3\xi(2it)}{\xi(1+2it)}E(z,3/2+it)\\ &+ \frac{3\xi(2it)^{2}}{\xi(1+2it)^{2}}E(z,3/2-it) + (\frac{\xi(2it)}{\xi(1+2it)})^{3}E(z,3/2-3it),        
    \end{aligned}
\end{equation*}
by Rankin-Selberg method we calculate
\[\int_{\mathbb{X}}\psi(z)E(z,3/2+it)\dd \mu z = \int_{0}^{\infty}y^{3/2+it}a_{\psi}(y)\frac{\dd y}{y^{2}}\]
where $a_{\psi}(y) = \int_{-1/2}^{1/2}\psi(x+iy)\dd x$. Since $\psi \in C_{c}^{\infty}(\mathbb{X})$ we have the derivative $a_{\psi}^{(j)}(y)$  is compactly supported in $(\sqrt{3}/2,\infty)$ . Thus by partial integral enough times we get
\[\int_{0}^{\infty}y^{3/2+it}a_{\psi}(y)\frac{\dd y}{y^{2}} \ll_{\psi,A} \frac{1}{(1+|t|)^{A}}.\]
Thus we get
\[I_{4} \ll_{\psi} (1+|t|)^{-100}.\]
\par
Now we consider $I_{2}$ and $I_{3}$. Note that we reduce the problem to $I_{2}$ and $I_{3}$ when $t_{j} , \tau \ll t^{\varepsilon}$ from following truncation.\par
We consider for any integral $\ell \geq 0$, 
\[(1/4+t_{j}^{2})^{\ell}\langle u_{j} , \psi\rangle  = \langle\Delta^{\ell} u_{j} , \psi\rangle = \langle u_{j} , \Delta^{\ell}\psi\rangle \ll_{\psi,\ell} 1\]
and
\[(1/4+t^{2})^{\ell}\langle E_{t} , \psi\rangle  = \langle\Delta^{\ell} E_{t} , \psi\rangle = \langle E_{t} , \Delta^{\ell}\psi\rangle \ll_{\psi,\ell} (1+|t|)^{3/8+\varepsilon}\]
from the sup norm estimate of $E_{t}$ \cite{HuangXu+2017+1355+1369}. Thus we get
\[\langle u_{j},\psi\rangle \ll_{\ell,\psi}t_{j}^{-2\ell}, \quad \langle E_{t},\psi\rangle \ll_{\ell,\psi}t^{-2\ell}.\]
We use the explicit formula proposition \ref{explicitMaass}, \ref{explicitEisen} and use the convexity bound of $L$-functions we get at least for a fixed large constant $B > 0$ such that
\[\langle u_{j} , E_{t}^{3} \rangle \ll (1+|t_{j}t|)^{B} \quad ,\langle E_{\tau} , E_{t}^{3} \rangle_{reg} \ll (1 + |\tau t|)^{B}.\]
Thus we easily truncate the sum and integral at $\tau,t_{j} \ll t^{\varepsilon}$ and the remainder is  an arbitrary  power saving of $t$. Thus
\begin{equation}
    \begin{aligned}
         &\sum\limits_{j \geq 1}\langle u_{j}, \psi \rangle \langle u_{j},E_{t}^{3}\rangle_{reg} +\frac{1}{4\pi}\int_{\mathbb{R}}\langle E_{\tau} , \psi\rangle \langle E_{\tau} ,  E_{t}^{3}\rangle_{reg}\dd \tau\\
     & = \sum\limits_{t_{j} \leq t^{\varepsilon}}\langle u_{j}, \psi \rangle \langle u_{j},E_{t}^{3}\rangle_{reg} +\frac{1}{4\pi}\int_{|\tau|\leq t^{\varepsilon}}\langle E_{\tau} , \psi\rangle \langle E_{\tau} ,  E_{t}^{3}\rangle_{reg}\dd \tau +  \mathcal{O}_{\psi}(t^{-100})\\
     & \ll_{\psi, \varepsilon} t^{-1/3+\varepsilon}
    \end{aligned}
\end{equation}

from Proposition \ref{Maasscase} and Proposition \ref{Eisensteincase}.

\section{\label{sec:6}The proof of Theorem \ref{JointEEg} ($t_{g} \geq  2t^{\varepsilon^{\prime}}$)}
In this section, we prove Theorem \ref{JointEEg}. After spectral decomposition and a natural truncation we obtain for any $\varepsilon^{\prime} > 0$
\begin{equation}\label{EEg}
    \begin{aligned}
         \int_{\Gamma\backslash\mathbb{H}}
    \psi(z){E_{t}^{\star}}(z)^{2}g(z)\dd \mu z & =
    \overline{c_{t}}^{2} \langle \psi , \frac{3}{\pi}\rangle\langle 1, E_{t}^2g\rangle
    + \delta_{t_g\leq t^{\varepsilon^{\prime}}}\langle \psi , g\rangle  \langle 1, {E_{t}^{\star}}^2 g^2\rangle\\
    & \quad +
    \overline{c_{t}}^{2}\sum_{\substack{t_k\leq \max(t, t_g)^{\varepsilon^{\prime}}\\ \phi_k\neq g}}\langle \psi , \phi_k\rangle \langle \phi_k, E_{t}^2g \rangle\\
    &\quad +
    \overline{c_{t}}^{2}\frac{1}{4\pi}\int_{|y|\leq \max(t, t_g)^{\varepsilon^{\prime}}}\langle \psi , E_{y}\rangle \langle E_{y}, E_{t}^2 g\rangle \dd y\\
    & \quad\quad + \mathcal{O}(\max(t, t_g)^{-A}).
    \end{aligned}
\end{equation}
 We divide the range into three parts(in fact depending on the Gamma factor)  $2t^{\varepsilon^{\prime}} \leq t_{g} \leq 2t - t_{g}^{\varepsilon}$, $ 2t - t_{g}^{\varepsilon}\leq t_{g} \leq 2t + t_{g}^{\varepsilon}$ and $t_{g} \geq 2t + t_{g}^{\varepsilon}$.\par
In the range $t_{g} \geq 2t^{\varepsilon^{\prime}}$, we have the second constant term is zero. 
So we mainly estimate the three terms
 \[
 \langle 1, E_{t}^2g\rangle , \langle \phi_k, E_{t}^2 g \rangle , \langle E_{y}, E_{t}^2g\rangle
 \]
when $t_k\ll \max(t, t_g)^{o(1)}$ and $|y|\ll \max(t, t_g)^{o(1)}$ in this section.

For the first term  we have the following proposition.
\begin{proposition}For any $\varepsilon > 0$,
if $t_{g} \geq 2t + t_{g}^{\varepsilon}$, then
\[\langle 1, E_{t}^2g\rangle  \ll t_{g}^{-100},\]
and if $ 2t - t_{g}^{\varepsilon}\leq t_{g} \leq 2t + t_{g}^{\varepsilon}$, then
    \[\langle 1, E_{t}^{2}g\rangle \ll_{\varepsilon} \frac{1}{|t_{g}|^{1/6+\delta-\varepsilon}}.\]
If $t_{g} \leq 2t - t_{g}^{\varepsilon}$, then
\begin{equation}
   \langle 1, E_{t}^{2}g\rangle \ll_{\varepsilon} \frac{t^{1/12+\varepsilon}}{(1 +|2t-t_{g}|)^{1/4}|t_{g}|^{1/6}} 
\end{equation}
or
\begin{equation}
    \langle 1, E_{t}^{2}g\rangle \ll_{\varepsilon} \frac{t^{\varepsilon}}{(1 +|2t-t_{g}|)^{\delta}(1 +|2t+t_{g}|)^{\delta}|t_{g}|^{1/6}}.
\end{equation}
where $\delta$ is from Lemma \ref{MichelVenkatesh}.
\end{proposition}
\begin{proof}
By Rankin-Selberg method, we get

\begin{equation}
    \begin{aligned}
  \langle 1, E_{t}^2g\rangle  \ll \frac{|L(1/2+2it, g)L(1/2,g)|}{|\zeta(1+2it)|}H(t,t_{g})    
    \end{aligned}
\end{equation}
where 
\begin{equation*}
    \begin{aligned}
        H(t,t_{g}) &= \frac{|\Gamma(\frac{1/2+2it+it_{g}}{2})\Gamma(\frac{1/2+2it-it_{g}}{2})||\Gamma(\frac{1/2+ it_{g}}{2})|^{2}}{|\Gamma(\frac{1+2it}{2})|^{2}|\Gamma(\frac{1+2it_{g}}{2}) |} \\
        & \asymp (1 +|2t-t_{g}|)^{-1/4}(1 +|2t+t_{g}|)^{-1/4}|t_{g}|^{-1/2}\exp(-\frac{\pi}{2}Q(t,t_{g})).\\
    \end{aligned}
\end{equation*}
where 
\[Q(t,t_{g}) = \frac{|2t-t_{g}|}{2} + \frac{|2t+t_{g}|}{2} - 2|t|.\]
Hence if $t_{g} \geq 2t + t_{g}^{\varepsilon}$, we have 
$H(t,t_{g}) \ll_{\varepsilon} \exp(-\frac{\pi}{2}t_{g}^{\varepsilon})$. Then $\langle1,E_{t}^{2}g\rangle \ll t_{g}^{-100}$. And if $t_{g} \leq 2t + t_{g}^{\varepsilon}$, we have $H(t,t_{g}) \ll (1 +|2t-t_{g}|)^{-1/4}(1 +|2t+t_{g}|)^{-1/4}|t_{g}|^{-1/2}$. Thus
\[\langle 1, E_{t}^{2}g\rangle \ll_{\varepsilon} \frac{(t+t_{g})^{1/3+\varepsilon}t_{g}^{1/3+\varepsilon}}{(1 +|2t-t_{g}|)^{1/4}(1 +|2t+t_{g}|)^{1/4}|t_{g}|^{1/2}}\]
or
\[\langle 1, E_{t}^{2}g\rangle \ll_{\varepsilon} \frac{t_{g}^{1/3+\varepsilon}}{(1 +|2t-t_{g}|)^{\delta}(1 +|2t+t_{g}|)^{\delta}|t_{g}|^{1/2}}.\]
Above equations come from  subconvexity bound (\ref{hybridsub}) or  Lemma \ref{MichelVenkatesh}.
\end{proof}

For the  $\langle\phi_{k}g, E_{t}^{2}\rangle$, we regard it as regularized inner product and use Plancherel formula
\begin{equation}\label{discrete}
\begin{aligned}
        \langle\phi_{k}g, E_{t}^{2}\rangle &= \sum\limits_{j\geq 1}\langle u_{j} , g\phi_{k}\rangle\langle E_{t}^{2}, u_{j}\rangle + \frac{1}{4\pi}\int_{\mathbb{R}}\langle E_{\tau}, u_{k}g\rangle\langle E_{\tau},E_{t}^{2}\rangle_{reg}\dd \tau 
        + \lim\limits_{t^\prime \rightarrow 0} \langle g\phi_{k} ,\mathcal{E}_{E_{t}E_{t^{\prime}}}\rangle.
\end{aligned}
\end{equation}
Note $g\neq \phi_{k}$, then the last integral is  holomorphic at $t^{\prime} = 0$ and the main part in the last integral is 
\begin{equation}\label{regularized}
 \frac{\rho_{g}(1)\rho_{k}(1)\Lambda(1 ,\phi_{k}\times g)}{\xi(2)} \asymp \frac{L(1,\phi_{k}\times g)\exp(-\frac{\pi}{2}(|t_{k} + t_{g}| + |t_{k} - t_{g}| - t_{k} - t_{g}))}{L(1,\sym^{2}g)^{1/2} L(1,\sym^{2}\phi_{k})^{1/2}}.
 \end{equation}
Then they will decay rapidly from Stirling formula  when $|t_{g}-t_{k}| \geq t^{\varepsilon^{\prime}}$.

\begin{proposition}\label{lastsection}Assume $t_{g} \geq 2t^{\varepsilon^{\prime}}$. For any $\varepsilon  > 0$ and $0< \eta < 1$, we get
 \begin{equation}
    \begin{aligned}
  \langle\phi_{k}g, E_{t}^{2}\rangle       \ll_{\varepsilon^{\prime}} \left\{	
			\begin{array}{lr}	
   t_{g}^{-100} \quad & t_{g} \geq 2t +t_{g}^{\varepsilon}\\
		\frac{t^{\varepsilon}}{t^{1/9+\delta/3}}\quad &      2t - t_{g}^{\varepsilon} \leq t_{g} \leq 2t + t_{g}^{\varepsilon} \\
  
 \frac{t^{\varepsilon}}{t^{1/9+\delta/3}(1+ |2t -t_{g}|)^{1/6+\delta/3}} \quad &      (2-\eta)t \leq t_{g} \leq 2t - t_{g}^{\varepsilon} \\
  
   \frac{t^{\varepsilon}}{t_{g}^{1/6}t^{1/6}} \quad &      t^{2/3 
 } \leq t_{g} \leq (2-\eta)t    \\
 
   \frac{t^{\varepsilon}}{t_{g}^{1/6}t^{1/12}}  \quad &       t^{1/3}\leq t_{g} \leq t^{2/3
 } \\

     \frac{t^{\varepsilon} t_{g}^{1/12}}{t^{1/6}} \quad &     t_{g} \leq t^{1/3}    \\
			\end{array}
			\right.
    \end{aligned}
\end{equation}
\end{proposition}
\begin{proof}
By Rankin-Selberg method and Watson's formula we get
\begin{equation*}
   \sum\limits_{j\geq 1}\langle u_{j} , g\phi_{k}\rangle\langle E_{t}^{2}, u_{j}\rangle \ll \sum\limits_{j\geq 1}\frac{|L(1/2+2it,u_{j})|L(1/2,u_{j})L(1/2,u_{j}\times g \times \phi_{k})^{1/2}H(t_{j}, t, t_{g}, t_{k})}{|\zeta(1+2it)|^{2}L(1,\sym^{2}u_{j})^{3/2}L(1,\sym^{2}g)^{1/2}L(1,\sym^{2}\phi_{k})^{1/2}}
\end{equation*}
where
\begin{equation*}
    \begin{aligned}
H(t_{j}, t,t_{g} , t_{k}) &\asymp \frac{|\Gamma(\frac{1/2+2it+it_{j}}{2})\Gamma(\frac{1/2+2it-it_{j}}{2})||\Gamma(\frac{1/2+ it_{j}}{2})|^{2}\prod\limits_{\pm_1}\prod\limits_{\pm_2}|\Gamma(\frac{1/2+it_{j}\pm_1 t_{g} \pm_2 t_{k}}{2})|}{|\Gamma(\frac{1+2it}{2})|^{2}|\Gamma(\frac{1+2it_{j}}{2}) |^{2}|\Gamma(\frac{1+2it_{g}}{2})\Gamma(\frac{1+2it_{k}}{2})|} \\
        & \asymp (1 +|2t-t_{j}|)^{-1/4}(1 +|2t+t_{j}|)^{-1/4}|t_{j}|^{-1/2}\prod_{\pm_1}\prod_{\pm_2}(1 + |t_{j} \pm_1 t_{g} \pm_2 t_{k}|)^{-1/4}\\
        & \quad \quad \exp(-\frac{\pi}{2}Q(t_{j} , t,t_{g} , t_{k})).\\    
    \end{aligned}
\end{equation*}
where 
\begin{equation*}
    \begin{aligned}
Q(t_{j} , t,t_{g} , t_{k})) &= \frac{|2t-t_{j}|}{2} + \frac{|2t + t_{j}|}{2}
+ \frac{|t_{j} +t_{g} + t_{k}|}{2} + \frac{|t_{j} - t_{g} + t_{k}|}{2} \\
&\quad + \frac{|t_{j} + t_{g}  - t_{k}|}{2} + \frac{|t_{j} -t_{g} - t_{k}|}{2} - 2t - t_{j} -t_{g} - t_{k}.        
    \end{aligned}
\end{equation*}
For $Q(t_j, t, t_g, t_k)$, we have the following lemma.
\begin{lemma}[\cite{hua2024jointvaluedistributionheckemaass}, (2.28), (2.29), (2.30)] If $2t \leq t_{g} - t_{k}$, we have
\begin{equation}\label{small}
Q(t_j, t, t_g, t_k)=
\begin{cases}
t_g-t_k-t_j,     &0\leq t_j\leq 2t,\\
t_g-2t-t_k,       &2t < t_j\leq t_g-t_k,\\
t_j-2t, &t_g-t_k<t_j\leq t_g+t_k,\\
2t_j-2t-t_g-t_k,  &t_g+t_k<t_j.\\
\end{cases}
\end{equation}
    If $t_g-t_k<2t\leq t_g+t_k$, we have
\begin{equation}\label{middle}
Q(t_j, t, t_g, t_k)=\begin{cases}
t_g-t_k-t_j,     &0\leq t_j\leq t_g-t_k,\\
0,       &t_g-t_k<t_j\leq 2t,\\
t_j-2t, &2t<t_j\leq t_g+t_k,\\
2t_j-2t-t_g-t_k,  &t_g+t_k<t_j\\
\end{cases}
\end{equation}
and if $t_g+t_k<2t$, we have
\begin{equation}\label{large}
Q(t_j, t, t_g, t_k)=\begin{cases}
t_g-t_k-t_j,     &0\leq t_j\leq t_g-t_k,\\
0,       &t_g-t_k<t_j\leq t_g+t_k,\\
t_j-t_g-t_k, &t_g+t_k<t_j\leq 2t,\\
2t_j-2t-t_g-t_k,  &2t<t_j.\\
\end{cases}
\end{equation}
\end{lemma}

Therefore, when  $t_{g} \geq 2t + t_{g}^{\varepsilon}$, we get
\[ \sum\limits_{j\geq 1}\langle u_{j} , g\phi_{k}\rangle\langle E_{t}^{2}, u_{j}\rangle \ll t_{g}^{-100}.\]
\par

So  when $2t - t_{g}^{\varepsilon} \leq t_{g} \leq 2t + t_{g}^{\varepsilon}$ and when  $ t_{g} \leq 2t - t_{g}^{\varepsilon}$, we will truncate the integral in
$ t_{g} - 2 t^{\varepsilon}\leq t_{j} \leq t_{g} + 2t^{\varepsilon}$. Thus,
\begin{equation}\label{CSsup1}
    \begin{aligned}
         \sum\limits_{j\geq 1}\langle u_{j} , g\phi_{k}\rangle\langle E_{t}^{2}, u_{j}\rangle \ll \sum\limits_{ |t_{j} - t_{g}| \leq 2t^{\varepsilon}}|\langle u_{j},g \phi_{k}\rangle\langle E_{t}^{2} , u_{j}\rangle|.
    \end{aligned}
\end{equation}
Then we obtain
\begin{equation}\label{Bound1}
    \begin{aligned}
 \sum\limits_{j\geq 1}\langle u_{j} , g\phi_{k}\rangle\langle E_{t}^{2}, u_{j}\rangle &\ll       (\sum\limits_{j \geq1}| \langle u_{j},g \phi_{k}\rangle|^{2})^{1/2} (\sum\limits_{ |t_{j} - t_{g}| \leq 2t^{\varepsilon}}|\langle E_{t}^{2} , u_{j}\rangle|^{2} )^{1/2}.\\
    \end{aligned}
\end{equation}
For the first sum by using Bessel inequality and the sup norm bound of $g$, $\|\phi_{k}\|_{\infty} \ll t_{k}^{1/4} \ll t^{\varepsilon}$ we get 
\begin{equation}\label{CSsup3}
\sum\limits_{j\geq 1}|\langle u_{j},g\phi_{k}\rangle|^{2} \leq \langle \phi_{k}g, \phi_{k}g\rangle \ll t^{\varepsilon}.    
\end{equation}

And the second sum is bounded by 
\begin{equation}\label{nonconductordrop}
    \begin{aligned}
 t^{\varepsilon^{\prime}}&\sum\limits_{ |t_{j} - t_{g}| \leq 2t^{\varepsilon}}\frac{L(1/2,u_{j})^{2}|L(1/2 + 2it,u_{j})|^{2}}{t_{j}\prod\limits_{\pm}(1 + |t_{j} \pm 2t|)^{1/2}}\\
 & \ll \frac{t^{\varepsilon}}{t_{g}t^{1/2}(1+ |2t -t_{g}|)^{1/2}}\sum\limits_{ |t_{j} - t_{g}| \leq 2t^{\varepsilon}}L(1/2,u_{j})^{2}|L(1/2 + 2it,u_{j})|^{2}\\
 & \ll\frac{t^{\varepsilon}}{t_{g}t^{1/2}(1+ |2t -t_{g}|)^{1/2}}[\sum\limits_{ |t_{j} - t_{g}| \leq 2t^{\varepsilon}}L(1/2,u_{j})^{4}]^{1/2}[\sum\limits_{ |t_{j} - t_{g}| \leq t^{\varepsilon}}|L(1/2 + 2it,u_{j})|^{4}]^{1/2} \\
    \end{aligned}
\end{equation}
By Jutila-Motohashi's work we have $\sum\limits_{ |t_{j} - t_{g}| \leq t^{\varepsilon}}L(1/2,u_{j})^{4} \ll t_{g}^{4/3+\varepsilon}$ and  $\sum\limits_{ |t_{j} - t_{g}| \leq t^{\varepsilon}}|L(1/2 + 2it,u_{j})|^{4} \ll t^{4/3+\varepsilon}$  if  $t_{g}\gg t^{2/3}$. By hybrid subconvexity bound and second moment estimate $\sum\limits_{ |t_{j} - t_{g}| \leq t^{\varepsilon}}|L(1/2 + 2it,u_{j})|^{4} \ll t^{5/3+\varepsilon}$ or $t^{4/3+\varepsilon}t_{g}$ . \par
Then we have
\begin{equation}
     \sum\limits_{j\geq 1}\langle u_{j} , g\phi_{k}\rangle\langle E_{t}^{2}, u_{j}\rangle \ll \left\{	
			\begin{array}{lr}		
			\frac{t_{g}^{2/3}t^{2/3}t^{\varepsilon}}{t_{g}t^{1/2}(1+ |2t -t_{g}|)^{1/2}}	 \quad &      t^{2/3  
 } \ll t_{g} \leq 2t+ t^{\varepsilon}    \\
   \frac{t_{g}^{2/3}t^{5/6}t^{\varepsilon}}{t_{g}t^{1/2}(1+ |2t -t_{g}|)^{1/2}}  \quad &      t^{1/3} \leq t_{g} \leq  t^{2/3  
 } \\
    \frac{t_{g}^{2/3}t^{2/3}t_{g}^{1/2}t^{\varepsilon}}{t_{g}t^{1/2}(1+ |2t -t_{g}|)^{1/2}}\quad &      t_{g} \leq t^{1/3}.    \\
			\end{array}
			\right.
\end{equation}
On the other hand, we consider the conductor dropping case that $|2t-t_{g}|$ is little. We use another H\"older inequality get 

\begin{equation}\label{conductordrop}
    \begin{aligned}
     \sum\limits_{j\geq 1}\langle u_{j} , g\phi_{k}\rangle\langle E_{t}^{2}, u_{j}\rangle
        &\ll\frac{t^{\varepsilon}}{t_{g}t^{1/2}(1+ |2t -t_{g}|)^{1/2}}[\sum\limits_{ |t_{j} - t_{g}| \leq t^{\varepsilon}}L(\frac{1}{2},u_{j})^{3}]^{2/3}[\sum\limits_{ |t_{j} - t_{g}| \leq t^{\varepsilon}}|L(\frac{1}{2} + 2it,u_{j})|^{6}]^{1/3}\\
        &\ll \frac{t^{\varepsilon}t_{g}^{2/3}}{t_{g}t^{1/3+2\delta/3}(1+ |2t -t_{g}|)^{1/3+2\delta/3}}[\sum\limits_{ |t_{j} - t_{g}| \leq t^{\varepsilon}}|L(\frac{1}{2} + 2it,u_{j})|^{4}]^{1/3}\\
        & \ll \frac{t^{\varepsilon}t_{g}^{2/3}t^{4/9}}{t_{g}t^{1/3 + 2\delta/3}(1+ |2t -t_{g}|)^{1/3+2\delta/3}}.
    \end{aligned}
\end{equation}
In conclusion we have
$ \sum\limits_{j\geq 1}\langle u_{j} , g\phi_{k}\rangle\langle E_{t}^{2}, u_{j}\rangle$ is bounded by
\begin{equation}
    \begin{aligned}
      \left\{	
			\begin{array}{lr}		
		\frac{t^{\varepsilon}t_{g}^{1/3}t^{2/9}}{t_{g}^{1/2}t^{1/6 + \delta/3}} \asymp  \frac{t^{\varepsilon}}{t^{1/9+\delta/3}}\quad &      2t - t_{g}^{\varepsilon} \leq t_{g} \leq 2t + t_{g}^{\varepsilon} \\
  
  \frac{t^{\varepsilon}t_{g}^{1/3}t^{2/9}}{t_{g}^{1/2}t^{1/6+\delta/3}(1+ |2t -t_{g}|)^{1/6+\delta/3}} \asymp \frac{t^{\varepsilon}}{t^{1/9+\delta/3}(1+ |2t -t_{g}|)^{1/6+\delta/3}} \quad &      (2-\eta)t \leq t_{g} \leq 2t - t_{g}^{\varepsilon} \\
  
  \frac{t_{g}^{1/3}t^{1/3}t^{\varepsilon}}{t_{g}^{1/2}t^{1/4}(1+ |2t -t_{g}|)^{1/4}}\asymp \frac{t^{\varepsilon}}{t_{g}^{1/6}t^{1/6}} \quad &      t^{2/3} \leq t_{g} \leq (2-\eta)t    \\
 
   \frac{t_{g}^{1/3}t^{5/12}t^{\varepsilon}}{t_{g}^{1/2}t^{1/4}(1+ |2t -t_{g}|)^{1/4}}\asymp\frac{t^{\varepsilon}}{t_{g}^{1/6}t^{1/12}}  \quad &       t^{1/3}\leq t_{g} \leq t^{2/3} \\

    \frac{t_{g}^{1/3}t^{1/3}t_{g}^{1/4}t^{\varepsilon}}{t_{g}^{1/2}t^{1/4}(1+ |2t -t_{g}|)^{1/4}}\asymp \frac{t^{\varepsilon} t_{g}^{1/12}}{t^{1/6}} \quad &      t_{g} \leq t^{1/3}    \\
			\end{array}
			\right.
    \end{aligned}
\end{equation}
Note that if conductor drop in equation \ref{nonconductordrop} we give $\mathcal{O}(t^{-1/12})$ and in equation \ref{conductordrop} we optimal $\mathcal{O}(t^{-1/9-\delta/3})$. 
The continuous part is enough from the nice bound of Riemann zeta function so we complete the proof.

\end{proof}
It is similar to estimate the inner product $\langle E_{y}, E_{t}^2g\rangle$, so in fact, we prove the Theorem \ref{JointEEg} when $t_{g}\geq 2t^{\varepsilon^{\prime}}$.\par

\section{\label{sec:7}The proof of Theorem \ref{JointEEg} 
 ($t_{g} \leq  2t^{\varepsilon^{\prime}}$)}
In this section, we mainly study the second constant term in equation (\ref{EEg})
\[\langle g^{2},{E_{t}^{\star}}^{2}\rangle = \overline{c_{t}}^{2}\langle g^{2},E_{t}^{2}\rangle\]
which will contribution a $\log$ contribution.
By Plancherel formula we get
\begin{equation*}
\begin{aligned}
        \langle g^{2}, E_{t}^{2}\rangle &= \sum\limits_{j\geq 1}\langle u_{j} , g^{2}\rangle\langle E_{t}^{2}, u_{j}\rangle + \frac{1}{4\pi}\int_{\mathbb{R}}\langle E_{\tau}, g^{2}\rangle\langle E_{\tau},E_{t}^{2}\rangle_{reg}\dd \tau 
        + \lim\limits_{\eta \rightarrow 0}\langle g^{2} , \mathcal{E}_{E(1/2 + it)E(1/2+ it + \eta)}\rangle.
\end{aligned}
\end{equation*}
The first two parts essentially can estimate as the same as previous case and are $\mathcal{O}(t^{-\frac{1}{6}+\varepsilon})$. We note that the regularized part $\lim\limits_{\eta \rightarrow 0}\langle g^{2} , \mathcal{E}_{E(1/2 + it)E(1/2+ it + \eta)}\rangle$ is important since we need calculate explicitly the four integrals and remove the possible singularity. \par
We have
\begin{proposition}\label{regularizedpart}For any $\varepsilon , \varepsilon^{\prime} > 0$, if $t_{g}\leq 2t^{\varepsilon^{\prime}}$ we have
    \[\langle g^{2},{E_{t}^{\star}}^{2}\rangle  = \frac{6}{\pi}(\log tt_{g} + \frac{L^{\prime}}{L}(1,\sym^{2}g)) + \mathcal{O}_{\varepsilon}(\log^{2/3+\varepsilon}t) + \mathcal{O}(t^{-\frac{1}{6}+\varepsilon}).\]
\end{proposition}
\begin{proof}
We obtain regularized part
\[\lim\limits_{\eta \rightarrow 0}\langle g^{2} , \mathcal{E}_{E(1/2 + it)E(1/2+ it + \eta)}\rangle\]
where $\eta$ is a small constant near $0$ such that $|\Re(\eta)| < 1/10$. Then
\begin{equation*}
    \begin{aligned}
 \mathcal{E}_{E(1/2 + it)E(1/2+ it + \eta)}(z)  = E(z,1+2it+\eta) &+ \frac{\xi(2it)}{\xi(1+2it)}E(z,1+\eta) + \frac{\xi(2it+2\eta)}{\xi(1+2it + 2\eta)}E(z,1-\eta)   \\
& \frac{\xi(2it)}{\xi(1+2it)}\frac{\xi(2it+2\eta)}{\xi(1+2it + 2\eta)}E(z,1-2it-\eta).
\end{aligned}
\end{equation*}
Note that 

\begin{equation*}
    \begin{aligned}
\langle g^{2} , E(z,1+2it)\rangle\asymp & \frac{L(1+2it,\sym^{2}g)\zeta(1+2it)}{L(1,\sym^{2}g)\zeta(2+4it)}\\
&\cdot \frac{|\Gamma(1/2+it)|^{2}|\Gamma(1/2+it +it_{g})\Gamma(1/2+it - it_{g})|}{|\Gamma(1/2+it_{g})|^{2}|\Gamma(1 + 2it)|}. 
    \end{aligned}
\end{equation*}
The Gamma factors will be 
\[\ll \frac{1}{t^{1/2}}\exp(-\frac{\pi}{2}(|t+t_{g}| + |t- t_{g}| - 2t_{g})) \ll e^{-t}.\]
Hence, this part is exponentially decreasing and the same argument to the last term.\par
Now our main work is calulate the second and the third term. Note that $\eta = 0$ is the singularity of these two functions but the residue exactly
has opposite sign. We still need the zero-th coefficient in Laurent expansion which is the limit value. We write
\[\frac{\xi(2it)}{\xi(1+2it)}\langle E(z,1+\eta)  , g^{2} \rangle =\frac{|\rho_{g}(1)|^{2}\xi(2it)}{\xi(1+2it)}\frac{\Lambda(1+\eta , \sym^{2}g)\xi(1+\eta)}{\xi(2+2\eta)}  = \xi(1+\eta)F_{1}(\eta)\]
and
\[\frac{\xi(2it + 2\eta)}{\xi(1+2it+ 2\eta)}\langle  E(z,1-\eta) ,  g^{2} \rangle =\frac{|\rho_{g}(1)|^{2}\xi(2it+ 2\eta)}{\xi(1+2it+ 2\eta)}\frac{\Lambda(1-\eta , \sym^{2}g)\xi(1-\eta)}{\xi(2-2\eta)}  = \xi(1-\eta)F_{2}(\eta).\]
We have expansion of $\xi(1+\eta)$ from classical theory of Riemann zeta function. 
\[\xi(1+\eta) = \frac{a_{-1}}{\eta} + a_{0} + a_{1}\eta + \mathcal{O}(\eta^{2})\]
\par
We calculate the coefficients in Taylar expansion of $F_{1}(\eta)$ and $F_{2}(\eta)$ explicitly up to $\mathcal{O}(\eta^{2})$ (the constant term we write $A_{0}$ and $B_{0}$) that is 
\[F_{1}(\eta) = \frac{|\rho_{g}(1)|^{2}\xi(2it)}{\xi(1+2it)}\frac{\Lambda(1, \sym^{2}g)}{\xi(2)} + A\eta  + \mathcal{O}(\eta^{2})\]
where 
\[A = \frac{|\rho_{g}(1)|^{2}\xi(2it)}{\xi(1+2it)} \lim\limits_{\eta \rightarrow 0}[\frac{\Lambda(1+\eta , \sym^{2}g)}{\xi(2+ 2\eta)}]^{\prime} = \frac{|\rho_{g}(1)|^{2}\xi(2it)}{\xi(1+2it)} \frac{\Lambda^{\prime}(1 , \sym^{2}g)\xi(2) - 2\Lambda(1)\xi^{\prime}(2)}{\xi(2)^{2}}\]
and 
\[F_{2}(\eta) = \frac{|\rho_{g}(1)|^{2}\xi(2it)}{\xi(1+2it)}\frac{\Lambda(1, \sym^{2}g)}{\xi(2)} + B\eta  + \mathcal{O}(\eta^{2})\]
where
\[B = |\rho_{g}(1)|^{2} \lim\limits_{\eta \rightarrow 0}[\frac{\xi(2it+2\eta)\Lambda(1-\eta , \sym^{2}g)}{\xi(1+2it + 2\eta)\xi(2- 2\eta)}]^{\prime}\]
then
\[B = |\rho_{g}(1)|^{2}\frac{\xi(2)\xi(1+2it)[2\xi^{\prime}(2it)\Lambda(1) - \xi(2it)\Lambda^{\prime}(1)] - \xi(2it)\Lambda(1)[2\xi^{\prime}(1+2it)\xi(2) - 2 \xi(1+2it)\xi^{\prime}(2)]}{\xi(1+2it)^{2}\xi(2)^{2}}.\]
We know the limit value we need is the conjugation of 
\[a_{-1}A + a_{0}A_{0} -a_{-1}B + a_{0}B_{0}\]
and by Stirling formula and
\[|\rho_{g}(1)|^{2} = \frac{\cosh \pi t_{g}}{2L(1,\sym^{2}g)} = \frac{\pi}{2\Gamma(1/2-it_{g})\Gamma(1/2+it_{g})L(1,\sym^{2}g)} = (2\Lambda(1,\sym^{2}g))^{-1}\]
we get
\[a_{0}A_{0} + a_{0}B_{0} = a_{0}\frac{\xi( 2it)}{\xi(2)\xi(1+2it)}.\]
For
\begin{equation}
    \begin{aligned}
  A - B &= \frac{\xi^{\prime}(1- 2it)}{\xi(1+2it)\xi(2)} + \frac{\xi(2it)}{\xi(1+2it)\xi(2)}[\frac{\Lambda^{\prime}(1,\sym^{2}g)}{\Lambda(1,\sym^{2}g)}  + \frac{\xi^{\prime}(1+2it)}{\xi(1+2it)}] - \frac{2\xi(2it)\xi^{\prime}(2)}{\xi(1+2it)\xi(2)^{2}}  \\
  & = \frac{\xi(2it)}{\xi(1+2it)\xi(2)}[\frac{\xi^{\prime}(1- 2it)}{\xi(1-2it)} + \frac{\Lambda^{\prime}(1,\sym^{2}g)}{\Lambda(1,\sym^{2}g)}  + \frac{\xi^{\prime}(1+2it)}{\xi(1+2it)} - \frac{2\xi^{\prime}(2)}{\xi(2)}].
    \end{aligned}
\end{equation}
Thus,
\[\overline{c_{t}}^{2}\lim\limits_{\eta \rightarrow 0}\langle g^{2} , \mathcal{E}_{E_{1/2 + it}E_{1/2+ it + \eta}}\rangle = \frac{1}{\xi(2)}\overline{[\frac{\xi^{\prime}(1- 2it)}{\xi(1-2it)} + \frac{\Lambda^{\prime}(1,\sym^{2}g)}{\Lambda(1,\sym^{2}g)}  + \frac{\xi^{\prime}(1+2it)}{\xi(1+2it)} - \frac{2\xi^{\prime}(2)}{\xi(2)} + a_{0}]}.\]
We have well-known result
\[\frac{\xi^{\prime}(1+2it)}{\xi(1+2it)} = -\frac{1}{2}\log\pi + \frac{1}{2}\frac{\Gamma^{\prime}(1+2it)}{\Gamma(1+2it)} + \frac{\zeta^{\prime}(1+2it)}{\zeta(1+2it)} = \frac{1}{2}\log t + \mathcal{O}(\log^{\frac{2}{3}+\varepsilon}t).\]
in the paper \cite[(2.22)]{spinu2003l4}
and by Stirling formula we have
\begin{equation}
    \begin{aligned}
        \frac{\Lambda^{\prime}(1,\sym^{2}g)}{\Lambda(1,\sym^{2}g)} &= -\frac{3}{2}\log\pi +\frac{1}{2}\frac{\Gamma^{\prime}(1/2)}{\Gamma(1/2)} + \frac{1}{2}\frac{\Gamma^{\prime}(1/2+it_{g})}{\Gamma(1/2+it_{g})} + \frac{1}{2}\frac{\Gamma^{\prime}(1/2- it_{g})}{\Gamma(1/2- it_{g})} + \frac{L^{\prime}}{L}(1,\sym^{2}g)\\
        & = \log t_{g} + \frac{L^{\prime}}{L}(1,\sym^{2}g) + \mathcal{O}(1).
    \end{aligned}
\end{equation}
Thus, from $\xi(2) = \pi/6$ the regularized part of $\langle g^{2} , {E_{t}^{\star}}^{2}\rangle$ will give
\[\frac{6}{\pi}[\log tt_{g} + \frac{L^{\prime}}{L}(1,\sym^{2}g) + \mathcal{O}(\log^{2/3+\varepsilon}t)].\]
Hence, we complete the proof of Proposition \ref{regularizedpart}.
\end{proof}
The other contributions in this case are almost the same as last section expect the regularized part in equation \ref{discrete}. This regularized part appears since the main part (equation \ref{regularized}) is not decreasing rapidly individually. But we can still
control their contributions in the following Lemma \ref{regulariedcontribution}.

Recall the regularized part in the discrete part is
\[\overline{c_{t}}^{2}\lim\limits_{t^\prime \rightarrow 0} \langle g\phi_{k} ,\mathcal{E}_{E_{t}E_{t^{\prime}}}\rangle = \frac{\rho_{g}(1)\rho_{k}(1)\Lambda(1 ,\phi_{k}\times g)}{\xi(2)} + \mathcal{O}(\exp(-t)).\]
Then they contribute
\[ \sum_{\substack{t_k\leq t^{\varepsilon^{\prime}}\\ \phi_k\neq g}}\langle \psi , \phi_k\rangle \frac{\rho_{g}(1)\rho_{k}(1)\Lambda(1 ,\phi_{k}\times g)}{\xi(2)} + \mathcal{O}(\exp(-t)). \]
The main term above is define as $A(\psi,t,t_{g})$. Similarly, the regularized part in the continuous part contributes
\[\frac{1}{4\pi}\int_{|y|\leq t^{\varepsilon^{\prime}}}\langle\psi,E_{y}\rangle  \frac{\rho_{g}(1)\rho_{y}(1)\Lambda(1 ,E_{y}\times g)}{\xi(2)}\dd y + \mathcal{O}(\exp(-t)).\]
And the main term above is define by $B(\psi,t,t_{g})$. The following lemma will give the bound of  $A(\psi,t,t_{g}), B(\psi,t,t_{g})$
\begin{lemma}\label{regulariedcontribution}
   If $t_{g} \leq 2t^{\varepsilon^{\prime}}$, for any large integer $A$, we get
    \[A(\psi,t,t_{g}) \ll_{\psi} t_{g}^{-A}.\]
    And we have the same bound for $B(\psi,t,t_{g})$.
\end{lemma}
\begin{proof}
    Recall 
\[A(\psi,t,t_{g}) \asymp \sum_{\substack{t_k\leq t^{\varepsilon^{\prime}}\\ \phi_k\neq g}}\langle \psi , \phi_k\rangle\frac{L(1,\phi_{k}\times g)\exp(-\frac{\pi}{2} |t_{k} - t_{g}|)}{L(1,\sym^{2}g)^{1/2} L(1,\sym^{2}\phi_{k})^{1/2}}.\]
Because the rapidly decreasing property of $\langle \psi,\phi_{k} \rangle$ we consider two sums
\[A(\psi,t,t_{g}) \ll \sum_{\substack{|t_{k}-t_{g}|\leq t_{g}/2 \\t_k\leq t^{\varepsilon^{\prime}}\\ \phi_k\neq g}} + \sum_{\substack{|t_{k}-t_{g}|\geq t_{g}/2\\ t_k\leq t^{\varepsilon^{\prime}}\\ \phi_k\neq g}}.\]
The first one will be bounded by 
\[(\log t_{g})^{1/2}\sum_{\substack{|t_k - t_{g}|\leq t_{g}/2\\ \phi_k\neq g}}\frac{1}{t_{g}^{A}}t_{g}^{\varepsilon} \ll t_{g}^{-A +3}\]
for any large integer $A$.\par
 And the second we bound by
\[(\log t_{g})^{1/2}\sum_{\substack{t_k\leq t^{\varepsilon^{\prime}}\\ \phi_k\neq g}}\frac{1}{t_{k}^{100}}\max\{t_{g},t_{k}\}^{\varepsilon}\exp(-\frac{\pi}{2}\cdot\frac{t_{g}}{2}) \ll t_{g}^{-A}.\]
for any large integer $A$.\par
So we have 
\[A(\psi,t,t_{g}) \ll t_{g}^{-A}\]
for any large integer $A$.
\end{proof}

Hence we complete the proof of Theorem \ref{JointEEg}.


\section{\label{sec:8}Joint value distribution of Hecke-Maass cusp forms}
In this section, we will prove Theorem \ref{Jointffg}.

As the same as \cite[The proof of Theorem 1.4]{hua2024jointvaluedistributionheckemaass}, we can truncate the sum in the following way.
\begin{equation}\label{ffg}
    \begin{aligned}
         \int_{\Gamma\backslash\mathbb{H}}
    \psi(z)f^2(z)g(z)\dd \mu z & =
    \langle \psi , \frac{3}{\pi}\rangle\langle 1, f^2g\rangle
    + \delta_{t_g\ll t_f^{o(1)}}\langle \psi , g\rangle  \langle 1, f^2 g^2\rangle\\
    & \quad +
    \sum_{\substack{t_k\ll \max(t_f, t_g)^{o(1)}\\ u_k\neq g}}\langle \psi , u_k\rangle \langle u_k, f^2g \rangle\\
    &\quad +
    \frac{1}{4\pi}\int_{|t|\ll \max(t_f, t_g)^{o(1)}}\langle \psi , E_{t}\rangle \langle E_{t}, f^2g\rangle \dd t\\
    & \quad\quad + \mathcal{O}(\max(t_f, t_g)^{-A}).
    \end{aligned}
\end{equation}
 Note that in our range $|t_{f} - t_{g}| \leq t_{f}^{\theta}$ which $\theta < 2/3$, the second constant term vanishes. So we mainly consider three terms
 \[
 \langle 1, f^2g\rangle , \langle u_k, f^2g \rangle , \langle E_{t}, f^2g\rangle
 \]
when $t_k\ll \max(t_f, t_g)^{o(1)}$ and $|t|\ll \max(t_f, t_g)^{o(1)}$.
 
 for the first constant term we get
\begin{proposition} When $|t_{f} - t_{g}| \leq t_{f}^{\theta}$ which $\theta < 2/3$, we have
     \begin{equation}
     \begin{aligned}
\langle 1, f^{2}g\rangle  \ll 
\left\{	
			\begin{array}{lr}		
				\frac{1}{t_{f}^{1/6-\varepsilon}} \quad &\theta \leq 1/3\\
    t_{f}^{\varepsilon}\frac{t_{g}^{1/6}t_{f}^{\frac{1+\theta}{2}}}{t_{f}} \asymp \frac{1}{t_{f}^{\theta/2 - 1/3 -\varepsilon}}\quad  & 1/3 < \theta \leq 1/2.\\
    \frac{1}{t_{f}^{1/12- \varepsilon}}  \quad  & 1/2 \leq \theta < 1.\\
			\end{array}
			\right.\\
   \end{aligned}
 \end{equation}
\end{proposition}
\begin{proof}
By Watson's formula and Stirling formula, we have
 \begin{equation}
     \begin{aligned}
\langle 1, f^{2}g\rangle &\ll \frac{L(1/2,g)^{1/2}L(1/2, \sym^{2}f \times g)^{1/2}\exp(-\frac{\pi}{2}(|2t_{f} + t_{g}|/2 + |2t_{f} - t_{g}|/2 - 2t_{f}))}{L(1,\sym^{2}f)L(1,\sym^{2}g)^{1/2}t_{g}^{1/2}\prod\limits_{\pm}(1 + |t_{g} \pm 2t_{f}|)^{1/4}}     
\end{aligned}
 \end{equation}
from the Weyl bound of $\GL(2)$ $L$-functions, Lemma \ref{moment_L_functions_6} (the subconvexity range) and the final case is from convexity bound.\par
\end{proof}

Now we analysis the discrete spectrum part $\langle u_{k}g ,f^{2}\rangle$ in equation (\ref{ffg}). 
\begin{proposition} When $|t_{f} - t_{g}| \leq t_{f}^{\theta}$ which $\theta < 2/3$, we have
     \begin{equation}
     \begin{aligned}
\langle u_k, f^2g \rangle \ll 
\left\{	
			\begin{array}{lr}		
				\frac{t_{f}^{\varepsilon}t_{f}^{1/3}}{t_{g}^{2/3}} \quad &\theta \leq 1/3\\
    \frac{t_{f}^{\varepsilon}t_{f}^{\theta}}{t_{g}^{2/3}} \quad  & 1/3 < \theta \leq 1.\\
			\end{array}
			\right.
   \end{aligned}
 \end{equation}
\end{proposition}
\begin{proof}

At first, we  use Plancherel formula  to decompose the integral to three parts and the constant part is zero because $\langle1,u_{k}g\rangle = 0$. The  Maass forms part $\sum\limits_{j\geq 1}\langle u_{k}g, \phi_{j}\rangle\langle f^{2} , \phi_{j}\rangle$ is main term. We obtain 
$\sum\limits_{j\geq 1}\langle u_{k}g , \phi_{j}\rangle\langle f^{2}, \phi_{j}\rangle$ is bounded by
\begin{equation*}
    \sum\limits_{j\geq 1}\frac{L(1/2,\sym^{2}f \times \phi_{j})^{1/2}L(1/2,\phi_{j})L(1/2,u_{k}\times g \times \phi_{j})^{1/2}H(t_{j}, t, t_{g}, t_{k})}{L(1,\sym^{2}f)L(1,\sym^{2}u_{k})^{1/2}L(1,\sym^{2}g)^{1/2}L(1,\sym^{2}\phi_{j})}
\end{equation*}
and spectral part $H(t_{j}, t_{f} , t_{g}, t_{k})$ can give the exponential decay unless $|t_{j}-t_{g}|\leq \max\{t_{f} , t_{g}\}^{\varepsilon}$ (Readers can see the proof of Proposition \ref{lastsection} for more details). 
\begin{equation}
    \begin{aligned}
 \sum\limits_{j\geq 1}\langle u_{k}g, \phi_{j}\rangle\langle f^{2} , \phi_{j}\rangle &\ll       (\sum\limits_{j \geq1}| \langle u_{k}g , \phi_{j}\rangle|^{2})^{1/2} (\sum\limits_{ |t_{j} - t_{g}| \leq \max\{t_{f} , t_{g}\}^{\varepsilon}}|\langle f^{2} , \phi_{j}\rangle|^{2} )^{1/2}. \\
    \end{aligned}
\end{equation}
Recall the first sum is bounded by  $\mathcal{O}(t_{f}^{\varepsilon})$ from equations (\ref{CSsup1}) (\ref{Bound1}) (\ref{CSsup3}) in Section $\ref{sec:6}$. And the second sum is bounded by
\begin{equation}\label{mixmoment}
    \begin{aligned}
 (t_{f})^{\varepsilon^{\prime}}&\sum\limits_{ |t_{j} - t_{g}| \leq t_{f}^{\varepsilon}}\frac{L(1/2,\phi_{j})L(1/2, \sym^{2}f \times \phi_{j})}{t_{j}\prod\limits_{\pm}(1 + |t_{j} \pm 2t_{f}|)^{1/2}}\\
 & \ll \frac{t_{f}^{\varepsilon}t_{g}^{1/3}}{t_{g}t_{f}}\sum\limits_{ |t_{j} - t_{g}| \leq t_{f}^{\varepsilon}}L(1/2, \sym^{2}f \times \phi_{j})\\
 & \ll  \frac{t_{f}^{\varepsilon}t_{g}^{1/3}}{t_{g}t_{f}} \sum\limits_{ |t_{j} - t_{f}| \leq t_{f}^{\theta} + t_{f}^{\varepsilon}}L(1/2, \sym^{2}f \times \phi_{j})\\
 & \ll \left\{	
			\begin{array}{lr}		
				\frac{t_{f}^{\varepsilon}t_{f}^{1/3}}{t_{g}^{2/3}} \quad &\theta \leq 1/3\\
    \frac{t_{f}^{\varepsilon}t_{f}^{\theta}}{t_{g}^{2/3}} \quad  & 1/3 < \theta \leq 1.\\
			\end{array}
			\right.
    \end{aligned}
\end{equation}
Then we obtain
\begin{equation}
    \begin{aligned}
 \sum\limits_{j\geq 1}\langle u_{k}g, \phi_{j}\rangle\langle f^{2} , \phi_{j}\rangle \ll  
 \left\{	
			\begin{array}{lr}		
				\frac{t_{f}^{\varepsilon}t_{f}^{1/6}}{t_{g}^{1/3}} \quad &\theta \leq 1/3\\
    \frac{t_{f}^{\varepsilon}t_{f}^{\theta/2}}{t_{g}^{1/3}} \quad  & 1/3 < \theta \leq 1.\\
			\end{array}
			\right.
    \end{aligned}
\end{equation}
Because the range we care is  $t_{f}\sim t_{g}$. The estimate of the continuous spectrum part of $\langle u_{k}g , f^{2}\rangle$ is similar then we get the estimate of  $\langle u_{k}g , f^{2}\rangle$.\par 
\end{proof}

\begin{proposition} When $|t_{f} - t_{g}| \leq t_{f}^{\theta}$ which $\theta < 2/3$, we have
     \begin{equation}
     \begin{aligned}
\langle E_{t}, f^2g \rangle \ll 
\left\{	
			\begin{array}{lr}		
				\frac{t_{f}^{\varepsilon}t_{f}^{1/3}}{t_{g}^{2/3}} \quad &\theta \leq 1/3\\
    \frac{t_{f}^{\varepsilon}t_{f}^{\theta}}{t_{g}^{2/3}} \quad  & 1/3 < \theta \leq 1.\\
			\end{array}
			\right.
   \end{aligned}
 \end{equation}
\end{proposition}
\begin{proof}

For continuous spectrum part in equation (\ref{ffg})  $\langle E_{t}g, f^2\rangle$, the main part is $\sum\limits_{j\geq 1}\langle E_{t}g, \phi_{j}\rangle\langle f^{2} , \phi_{j}\rangle$. By using the same argument above we have
\begin{equation*}
    \begin{aligned}
   \sum\limits_{j\geq 1}\langle E_{t}g, \phi_{j}\rangle\langle f^{2} , \phi_{j}\rangle
    \ll t_{f}^{\varepsilon}[\sum\limits_{ |t_{j} - t_{g}| \leq t_{f}^{\varepsilon}}&\frac{L(1/2,\phi_{j})L(1/2, \sym^{2}f \times \phi_{j})}{t_{j}\prod\limits_{\pm}(1 + |t_{j} \pm 2t_{f}|)^{1/2}}]^{1/2}\\ &\cdot [\frac{1}{t_{g}}\sum\limits_{ |t_{j} - t_{g}| \leq t_{f}^{\varepsilon}}|L(1/2+it, g \times \phi_{j})|^{2} ]^{1/2}.
    \end{aligned}
\end{equation*}
The first term is exactly equation $(\ref{mixmoment})$ and for the second term we have a large sieve estimate $\sum\limits_{ |t_{j} - t_{g}| \leq t_{f}^{\varepsilon}}|L(1/2+it, g \times \phi_{j})|^{2} \ll t_{g}^{1+\varepsilon}$ thanks to the conductor dropping phenomenon. In  conclusion, we give the proof of continuous spectrum part in equation (\ref{ffg}). 
\end{proof}
Hence we prove Theorem \ref{Jointffg} .

\section*{Acknowledgements}

The author would like to thank  Prof. Bingrong Huang for his helpful discussions. He also wants to thank Liangxun Li for his comments.  He gratefully thanks to the referees for the constructive
 comments and recommendations.

\addcontentsline{toc}{section}{references}
\phantomsection

\end{document}